\documentclass{amsart}
\usepackage{amsmath, amsthm, amsfonts, mathtools, tikz-cd, amssymb}
\usepackage{enumitem}
\usepackage{caption}
\usepackage[all,cmtip]{xy}
\usepackage{mathabx}

\usepackage[pagebackref,pdfborder={0 0 0},urlcolor=black]{hyperref}
\hypersetup{pdfauthor={Kevin Li, Clara L\"oh, Marco Moraschini},%
  pdftitle={Bounded acyclicity and relative simplicial volume}}

\title[Bounded acyclicity and relative simplicial volume]%
      {Bounded acyclicity\\ and relative simplicial volume}
\author{Kevin Li}
\address{School of Mathematical Sciences, University of Southampton, Southampton SO17 1BJ, United Kingdom}
\email{kevin.li@soton.ac.uk}

\author{Clara L\"oh}
\address{Fakult\"at f\"ur Mathematik, Universit\"at Regensburg, 93040 Regensburg, Germany}
\email{clara.loeh@ur.de}

\author{Marco Moraschini}
\address{Dipartimento di Matematica, Universit{\`a} di Bologna, 40126 Bologna, Italy}
\email{marco.moraschini2@unibo.it}

\date{\today.\ \copyright{\ K.~Li, C.~L\"oh, M.~Moraschini 2022}.
  This work was partially supported by the CRC~1085 \emph{Higher Invariants}
  (Universit\"at Regensburg, funded by the DFG)}

\makeatletter
\@namedef{subjclassname@2020}{%
  \textup{2020} Mathematics Subject Classification}
\makeatother

\theoremstyle{definition}
\newtheorem{defn}{Definition}[section]
\newtheorem{ex}[defn]{Example}
\newtheorem{question}[defn]{Question}
\newtheorem{setup}[defn]{Setup}
\newtheorem*{ack}{Acknowledgements}
\theoremstyle{plain}
\newtheorem{thm}[defn]{Theorem}
\newtheorem{lem}[defn]{Lemma}
\newtheorem{prop}[defn]{Proposition}
\newtheorem{cor}[defn]{Corollary}

\theoremstyle{remark}
\newtheorem{rem}[defn]{Remark}

\numberwithin{equation}{section}

\newcommand{\IN}{\ensuremath\mathbb{N}}
\newcommand{\IZ}{\ensuremath\mathbb{Z}}
\newcommand{\IQ}{\ensuremath\mathbb{Q}}
\newcommand{\IR}{\ensuremath\mathbb{R}}
\let\N\IN
\let\Z\IZ
\let\Q\IQ
\let\R\IR

\newcommand{\calC}{\ensuremath\mathcal{C}}
\newcommand{\calU}{\ensuremath\mathcal{U}}
\newcommand{\calV}{\ensuremath\mathcal{V}}
\newcommand{\calH}{\ensuremath\mathcal{H}}
\newcommand{\calN}{\ensuremath\mathcal{N}}
\newcommand{\calG}{\ensuremath\mathcal{G}}
\newcommand{\calA}{\ensuremath\mathcal{A}}

\newcommand{\F}{\ensuremath\mathcal{F}}
\newcommand{\TR}{\ensuremath\mathsf{Tr}}
\newcommand{\FIN}{\ensuremath\mathsf{Fin}}
\newcommand{\AME}{\ensuremath\mathsf{Am}}
\newcommand{\EFG}{\ensuremath E_{\F}G}

\newcommand{\EFGH}{\ensuremath E_{\F}(G,\calH)}

\newcommand{\enum}{\rm{(\roman*)}}
\newcommand{\spann}[1]{{\ensuremath \langle{#1}\rangle}}
\newcommand{\into}{\ensuremath\hookrightarrow}
\DeclareMathOperator{\Hom}{Hom}
\DeclareMathOperator{\im}{im}
\DeclareMathOperator{\Stab}{Stab}
\DeclareMathOperator{\cd}{cd}
\DeclareMathOperator{\cat}{cat}
\DeclareMathOperator{\mult}{mult}

\DeclareMathOperator{\id}{id}
\DeclareMathOperator{\Ext}{Ext}
\DeclareMathOperator{\UBC}{UBC}
\DeclareMathOperator{\UUBC}{UUBC}
\DeclareMathOperator{\res}{res}
\DeclareMathOperator{\map}{map}
\DeclareMathOperator{\Group}{Group}
\DeclareMathOperator{\lf}{lf}
\DeclareMathOperator{\comp}{comp}
\DeclareMathOperator{\alt}{alt}
\newcommand{\Balt}{\ensuremath\ell^\infty_{\alt}}
\DeclareMathOperator{\ev}{ev}
\DeclareMathOperator{\sumop}{sum}
\newcommand{\linf}{\ensuremath\ell^\infty}
\newcommand{\graf}{\ensuremath\Gamma}
\newcommand{\gog}{\ensuremath\mathbb{G}}

\def\args{\,\cdot\,}

\def\qand{\quad\text{and}\quad}
\def\fa#1{%
  \forall_{#1}\;\;\;}
\def\exi#1{%
  \exists_{#1}\;\;\;}
\makeatletter
\def\bprod#1{%
  \begingroup%
   \if@display%
      \def\@tempa{\sideset{}{^b}\prod_{#1}}%
   \else%
      \def\@tempa{\prod_{#1}^b}%
   \fi%
   \@tempa\endgroup}
\makeatother
\def\sv#1{%
  \|#1\|}

\def\actson{%
  \curvearrowright}
\DeclareMathOperator{\lfop}{lf}
\def\svlf#1{%
  \|#1\|_{\lfop}}
\def\svlfR#1{%
  \|#1\|_{R,\lfop}}

\makeatletter
\def\l@subsection{\@tocline{2}{0pt}{2.5pc}{5pc}{}}
\makeatother

\def\admissible{weak\-ly con\-vex}
\def\admissibility{weak con\-vex\-i\-ty}

\DeclareMathOperator{\relcat}{mult}
\def\category{multiplicity}

\begin{document}

\begin{abstract}
  We provide new vanishing and glueing results for relative simplicial
  volume, following up on two current themes in bounded
  cohomology: The passage from amenable groups to boundedly acyclic
  groups and the use of equivariant topology.
  
  More precisely, we consider equivariant nerve pairs and relative
  classifying spaces for families of subgroups. Typically, we apply
  this to uniformly boundedly acyclic families of subgroups.  Our
  methods also lead to vanishing results for $\ell^2$-Betti numbers of
  aspherical CW-pairs with small relative amenable category and to a 
  relative version of a result by Dranishnikov and Rudyak
  concerning mapping degrees and the inheritance of freeness of
  fundamental groups.
\end{abstract}

\maketitle

\section{Introduction}

Bounded cohomology is defined as the cohomology of the bounded dual of
the singular or bar chain complex~\cite{vbc} and it has many
applications in group theory and geometry of manifolds. A fundamental
phenomenon is that bounded cohomology of amenable groups is trivial
(i.e., amenable groups are boundedly acyclic). On the other hand, the bounded
cohomology of negatively curved groups surjects onto ordinary
cohomology. In manifold topology, the simplicial volume of an oriented
compact manifold is a homotopy invariant defined as 
the $\ell^1$-seminorm of the $\R$-fundamental class~\cite{vbc}.
Using a duality argument, the simplicial volume can be expressed in terms of bounded cohomology.

We provide new vanishing results for relative simplicial volume,  
following up on two current themes in bounded cohomology:
\begin{itemize}
\item
  The passage from amenable groups to boundedly acyclic groups;
\item
  The use of equivariant topology, most notably of classifying
  spaces for families of subgroups.
\end{itemize}

A technical difficulty in the passage from amenable to boundedly
acyclic groups is that the class of amenable groups possesses
a large degree of uniformity when it comes to bounded cohomology.
This includes the fact that the class of amenable groups is closed
under subgroups and quotients and the fact that amenable groups
are not only boundedly acyclic, but uniformly boundedly acyclic.
Therefore, in the setting of boundedly acyclic groups, generalised
vanishing results for simplicial volume come with additional
uniformity and closure hypotheses.

As we aim at results for relative bounded cohomology and
relative simplicial volume, we adapt tools from equivariant
topology to this relative setting.

\subsection{Uniform bounded acyclicity}

Group actions with amenable stabilisers have proved to be a valuable
tool to compute bounded cohomology~\cite{Monod, burgermonodGAFA,
  BIuseful}. Similarly, also \emph{uniformly boundedly acyclic
actions} allow us to compute bounded cohomology, where the uniformity
refers to a uniform bound on the norms of primitives. Recently,
uniformly boundedly acyclic actions have been used to compute the
bounded cohomology of geometrically relevant
groups~\cite{fffclmm2,monodnariman}.

Let~$X$ be a path-connected space.
We say that a set of path-connected subspaces~$\calA$ of~$X$ is 
\emph{uniformly boundedly acyclic} [\emph{of order~$n$}] \emph{in~$X$} if 
the collection of all finite [resp.~$n$-fold] intersections of conjugates of 
the subgroups $$\bigl(\im(\pi_1(A\into X))\bigr)_{A\in \calA}$$ in~$\pi_1(X)$ is 
uniformly boundedly acyclic (Definition~\ref{def:ubacgroup}).
In the special case when the above groups are amenable, 
we also speak of an \emph{amenable} set of subspaces.
The issue of basepoints is addressed in Section~\ref{subsec:convs}. 
We have two geometric situations in which 
uniformly boundedly acyclic sets of subspaces lead to interesting
uniformly boundedly acyclic actions: Open covers and 
glueing loci of manifolds obtained by glueing manifolds with boundary.

\subsection{Vanishing via relative open covers}
\label{subsec:intro:vanishing}

Gromov~\cite{vbc} and Ivanov~\cite{ivanov}
established vanishing results for the comparison map (and thus for
simplicial volume) in the presence of amenable open covers with small
multiplicity.

Following the approach by L\"oh and Sauer~\cite{Loeh-Sauer19} through
equivariant nerves and classifying spaces for families,
we generalise these vanishing results in two directions.
First, we allow more general covers:
A cover~$\calU$ of~$X$ by path-connected open subsets
is \emph{uniformly boundedly acyclic} if
the underlying set of subsets of~$X$ is uniformly boundedly acyclic in~$X$.
Second, we adapt the setting to pairs of CW-complexes~$(X,A)$,
where~$A$ is $\pi_1$-injective in~$X$ (Theorem~\ref{thm:relative vanishing theorem}).
To this end, we introduce the notion of [\emph{weakly convex}] \emph{relative covers}
(Definition~\ref{defn:relative cover}).
Using equivariant nerve pairs and classifying spaces of group pairs for families, 
we obtain:

\begin{thm}[Corollary~\ref{cor:vanishing UBAc}]
\label{thm:intro:vanishing UBAc}
	Let~$(X,A)$ be a CW-pair with path-connected ambient space~$X$.
	Assume that~$A$ has only finitely many connected components,
	each of which is $\pi_1$-injective in~$X$.
	Let~$\calU$ be a relative cover of~$(X,A)$ that is uniformly boundedly acyclic.
	\begin{enumerate}[label=\enum]
		\item If~$\calU$ is weakly convex, then the comparison map
		\[
			\comp^k\colon H^k_b(X,A;\IR)\to H^k(X,A;\IR)
		\]
		vanishes in all degrees~$k\ge \mult_A(\calU)$.
		\item Let~$\nu\colon (X,A)\to (|N(\calU)|,|N_A(\calU)|)$ be a nerve map.
		If~$\calU$ is convex, then the comparison map~$\comp^*$ factors
		through~$\nu$:
		\[\begin{tikzcd}
			H^*_b(X,A;\IR)\ar{r}{\comp^*}\ar[dashed]{dr} & H^*(X,A;\IR) \\
			& H^*(|N(\calU)|,|N_A(\calU)|;\IR)\ar{u}[swap]{H^*(\nu;\IR)} \,.
		\end{tikzcd}\]
	\end{enumerate}
\end{thm}

Here $\mult_A(\calU)$ denotes the relative multiplicity of~$\calU$ 
with respect to~$A$ (Definition~\ref{def:relative:mult})
and the simplicial complex~$N_A(\calU)$ is a suitable subcomplex of
the nerve~$N(\calU)$ (Definition~\ref{defn:nerve pair}).

In the absolute case, Ivanov proved a similar vanishing theorem for
\emph{weakly boundedly acyclic covers} using spectral 
sequences~\cite{ivanov_bac_covers}. 
Our notion of uniformly boundedly acyclic covers is similar,
but the relation between the two is unclear (Remark~\ref{rem:Ivanov}).

Theorem~\ref{thm:intro:vanishing UBAc} applies in particular to
relative covers that are amenable. 
We introduce the \emph{relative amenable \category}
$\relcat_\AME(X,A)$ (Definition~\ref{defn:relative category})
as the minimal relative multiplicity of weakly convex
relative amenable covers of~$(X,A)$ by path-connected open subsets.

\begin{thm}[Corollary~\ref{cor:rel:van:thm:ame}]
\label{thm:intro:relvan}	
	Let~$(X,A)$ be a CW-pair with path-connected ambient space~$X$.
	Assume that~$A$ consists of finitely many connected components,
	each of which is $\pi_1$-injective in~$X$. Then the comparison map
	\[
		\comp^k\colon H^k_b(X,A;\IR)\to H^k(X,A;\IR)
	\]
	vanishes in all degrees~$k\ge \relcat_\AME(X,A)$.
	
	In particular, if~$(M,\partial M)$ is an oriented compact connected triangulable
	manifold with $\pi_1$-injective boundary components and
	$\relcat_\AME(M,\partial M)\le \dim(M)$, then the relative simplicial
	volume~$\|M,\partial M\|$ vanishes.
\end{thm} 

In the absolute case, every cover is a weakly convex relative cover and hence
$\relcat_\AME(X,\emptyset)$ is the minimal multiplicity of
amenable covers of~$X$. For a CW-complex~$X$, this coincides with the minimal
\emph{cardinality} of amenable covers of~$X$ by not necessarily path-connected subsets~\cite[Remark~3.13]{CLM20}.
The latter quantity is called the \emph{amenable category} $\cat_\AME(X)$ 
(Remark~\ref{rem:abscat}),
a notion that is modelled on the classical LS-category~\cite{CLOT03}.

As an application of Theorem~\ref{thm:intro:relvan}, we give an
alternative proof of a relative vanishing theorem, which is a
consequence of Gromov's vanishing theorem for non-compact manifolds
(Theorem~\ref{thm:rel:van:thm:LMR}).

Our methods for equivariant nerve pairs and relative classifying
spaces also lead to vanishing results for $\ell^2$-Betti numbers of
aspherical CW-pairs with small relative amenable \category\
(Theorem~\ref{thm:l2rel}). In the absolute case
(Corollary~\ref{cor:l2abs}), this recovers a result by
Sauer~\cite[Theorem~C]{Sauer09}.

\subsection{Glueings}

We adapt the additivity of relative simplicial volume for
glueings along amenable boundaries~\cite{vbc,BBFIPP,Kuessner}
to situations with boundedly acyclic boundaries. As we move away
from amenability, we lose control on the norm, and thus only
retain control on the vanishing behaviour.

\begin{thm}[Theorem~\ref{thm:bacglue}]
  Let $n\ge 3$ and $(M_i,\partial M_i)_{i\in I}$ be a finite
  collection of oriented compact connected $n$-manifolds.  Assume that
  every connected component of every boundary component~$\partial M_i$
  has boundedly acyclic fundamental group.  Let~$\calN$ be a set
  of~$\pi_1$-injective boundary components of the~$(M_i)_{i\in I}$ and let
  $(M,\partial M)$ be obtained from~$(M_i,\partial M_i)_{i\in I}$ by a
  pairwise glueing of the boundary components in~$\calN$.
	
  If~$\calN$, viewed as a set of subsets of~$M$, is uniformly boundedly
  acyclic of order~$n$ in~$M$, then the following are equivalent:
  \begin{enumerate}[label=\enum]
  \item We have $\|M,\partial M\|=0$;
  \item For all $i\in I$, we have $\|M_i,\partial M_i\|=0$.
  \end{enumerate} 
\end{thm}

\subsection{Mapping degrees}

One of the classical applications of simplicial volume is an a priori
estimate on mapping degrees~\cite{vbc, Thurston, FM:ideal}. In
contrast, the exact relation between mapping degrees and monotonicity
of (generalised) LS-category invariants is still wide
open~\cite{Rudyak_monotonicity, CLM20}.

In the absolute case, Eilenberg and Ganea showed that the LS-category
invariant for the family containing only the trivial subgroup of an 
aspherical space recovers the cohomological dimension of its fundamental group~\cite{eilenbergganea}.
Moreover, cohomological dimension one can be characterised in terms
of freeness. Thus, the monotonicity problem for (generalised) LS-category
leads to inheritance properties of fundamental groups under maps
of non-zero degree.

We use equivariant and group cohomological methods to establish the
following relative version (and a simplified proof) of a monotonicity
result by Dranishnikov and Rudyak for closed manifolds~\cite{Dranishnikov-Rudyak09}:

\begin{thm}[Corollary~\ref{cor:degree one relative}]
  Let $f\colon (M,\partial M)\to (N,\partial N)$ be a map between
  oriented compact connected manifolds of the same dimension with
  $\pi_1$-injective boundary components.  Let $\partial
  M=\coprod_{i=1}^m M_i$ and $\partial N=\coprod_{i=1}^n N_i$ be
  decompositions into connected components.  If $\deg(f)=\pm 1$
  and there exists a free group~$F_M$ 
  such that~$$\pi_1(M)\cong F_M\ast \Asterisk_{i=1}^m \pi_1(M_i)\,,$$ then
  there exists a free group~$F_N$ such that~$\pi_1(N)\cong F_N\ast \Asterisk_{i=1}^n
  \pi_1(N_i)$.  
\end{thm}        

For closed manifolds our approach also yields inheritance properties
for virtual freeness:

\begin{thm}[Corollary~\ref{cor:finmono}]
  Let $f\colon M\to N$ be a map between oriented closed connected
  manifolds of the same dimension.  If $\deg(f)\neq 0$ and $\pi_1(M)$
  is virtually free, then also~$\pi_1(N)$ is virtually free.
\end{thm}

\subsection{Conventions}\label{subsec:convs}

In this article, we adhere to the following conventions: 

Instead of the usual notion of families of sets, groups, modules, \dots,
we will speak of \emph{collections}; this is to avoid confusion with
the term ``families of subgroups''. I.e., a collection~$(H_i)_{i \in I}$
of groups [or sets, \dots] is a map~$I \to \Group$, $i \mapsto H_i$
from a set~$I$ to the class of all groups [or sets, \dots]. In particular,
collections can contain repetitions. 

Families of subgroups will only be closed under conjugation
but not necessarily under finite intersections or taking subgroups
(Definition~\ref{def:family}).

All groups will be discrete groups; in particular, we consider
bounded cohomology of discrete groups and
$G$-CW-complexes for discrete groups~$G$. 
The geometric realisation of $G$-simplicial complexes
will always be equipped with the $G$-CW-structure coming from
the barycentric subdivision (Example~\ref{ex:from:sc:to:cw}).

Given a compact manifold~$M$ with non-empty boundary, we say that $M$
has $\pi_1$-\emph{injective boundary} if every connected component
of~$\partial M$ is $\pi_1$-injective in~$M$. 

We usually refrain from spelling out basepoints for fundamental
groups.  Strictly speaking, fixing basepoints is necessary to make the
notion of the image of ``the'' fundamental group of a subspace in
``the'' fundamental group of an ambient space precise. However, we
will always deal with situations concerning conjugation-invariant
properties or concerning collections of all conjugates of such
subgroups. Therefore, all choices of basepoints would lead to the same
outcome.

We always work with open covers consisting of path-connected sets.
We explain in Remark~\ref{rem:abscat} why this condition is not
restrictive in our setting.

\begin{ack}
We would like to thank Wolfgang L\"uck for helpful comments on
classifying spaces of families.  We are grateful to George
Raptis for many useful discussions on bounded acyclicity, vanishing
results, and glueing formulae.  The first author thanks
his advisors Nansen Petrosyan and Ian Leary for several helpful conversations.
\end{ack}

 \tableofcontents

\section{Classifying spaces of group pairs for families of subgroups}

The goal of this section is to introduce classifying
spaces for families of subgroups in a relative setting.
We first recall classifying spaces for families of subgroups
and then explain the extension to the relative setting.

This is motivated by our geometric situations of topological
pairs~$(X,A)$ (e.g., manifolds with boundary), where often two classes
of subgroups of the fundamental group~$G\coloneqq \pi_1(X)$ will be
involved:
\begin{itemize}
\item A family~$\F$ of subgroups of~$G$, describing the allowed
  fundamental groups of sets in open covers of~$X$;
\item A collection~$\calH$ of subgroups of~$G$, coming from the
  fundamental groups of the components of~$A$.
\end{itemize}

\subsection{$G$-CW-complexes}

We briefly recall the definitions of $G$-CW-complexes and the
induction functor. For more background on $G$-CW-complexes we refer
the reader to the literature~\cite{lueck_TG,Lueck05}.

\begin{defn}[$G$-CW-complex]
  A \emph{$G$-CW-complex}~$Y$ is a~$G$-space equipped with
  a~$G$-invariant filtration
  \begin{equation}\label{eqn:G-CW filtration}
    \emptyset = Y^{(-1)}\subset Y^{(0)}\subset Y^{(1)}\subset \cdots\subset
    Y^{(n)}\subset \cdots\subset Y
  \end{equation}
  such that the following hold:
  \begin{itemize}
  \item $Y=\bigcup_{n\ge 0}Y^{(n)}$;
  \item $Y$ carries the weak topology with respect to the 
    filtration~\eqref{eqn:G-CW filtration};
  \item $Y^{(n)}$ is obtained from~$Y^{(n-1)}$ as a~$G$-pushout of the form
    \[\begin{tikzcd}
    \coprod_{i\in I_n}G/H_i\times S^{n-1}\ar{r}\ar{d} &
    Y^{(n-1)}\ar{d} \\
    \coprod_{i\in I_n}G/H_i\times D^n\ar{r} &
    Y^{(n)} \,.
    \end{tikzcd}\]
  \end{itemize}
  The subgroups~$H_i$ of~$G$ and their conjugates are called \emph{isotropy groups} of~$Y$.
  If all isotropy groups of~$Y$ are trivial, we also say that
  $Y$ is a \emph{free} $G$-CW-complex.
  
  A morphism of $G$-CW-complexes is a $G$-map.
\end{defn}

For example, the universal covering $\widetilde{X}$ of a
path-connected CW-complex $X$ is a free $\pi_1(X)$-CW-complex
with respect to the CW-structure inherited from~$X$.

\begin{ex}[Barycentric subdivision of $G$-simplicial complexes]\label{ex:from:sc:to:cw}
  Let $N$ be an (abstract) simplicial complex and let $G$ be a group
  acting on~$N$ via simplicial automorphisms.  Then the geometric
  realisation~$|N'|$ of the barycentric subdivision is a $G$-CW-complex,
  while $|N|$ need not be a $G$-CW-complex in general. The standard
  homeomorphism between the geometric realisations $|N| \to |N'|$ is a
  $G$-homeomorphism.  Therefore, $|N|$ admits a canonical structure as a
  $G$-CW-complex and we will always use this $G$-CW-structure.
\end{ex}

Given a subgroup $H \subset G$, there is a natural way to associate to
an $H$-CW-complex a $G$-CW-complex.

\begin{defn}[Induction]\label{defn:induction}
  Let $H$ be a subgroup of $G$. The \emph{induction (along the
    inclusion $H\subset G$)} is the functor
  \[
  G\times_H(-)\colon H\text{-CW-complexes}\to G\text{-CW-complexes}
  \]
  that assigns to an~$H$-CW-complex~$B$ the~$G$-CW-complex~$G\times_H
  B$, that is the quotient of~$G\times B$ by the (right)~$H$-action
  $(g,b)\cdot h=(gh,h^{-1}b)$. Here $G$ acts on~$G\times_H B$ by left
  multiplication.  We denote elements of $G\times_H B$ by $[g,b]$.
	
  For an~$H$-map~$f\colon B\to C$ between $H$-CW-complexes, the
  induced $G$-map $$G\times_H f\colon G\times_H B\to G\times_H C$$ is
  given by $G\times_H f([g,b])= [g,f(b)] $.
\end{defn}

The induction functor is left-adjoint to the restriction functor,
which associates to a $G$-CW-complex the same space viewed as an
$H$-CW-complex.

\subsection{Classifying spaces for families of subgroups}\label{sec:families}

We use the following (non-standard) convention for families of
subgroups:

\begin{defn}[Family of subgroups]\label{def:family}
  Let $G$ be a group and $\F$ be a non-empty set of subgroups of~$G$.
  We say that $\F$ is a \emph{family of subgroups} (or
  \emph{conjugation-closed family of subgroups}) of~$G$ if it is
  closed under conjugation.  We say that $\F$ is an
  \emph{intersection-closed family of subgroups} of~$G$ if it is
  closed under conjugation and taking finite intersections.
\end{defn}

\begin{ex}\label{ex:family:subgroups}
  The following are examples of families of subgroups:
  \begin{enumerate}[label=\enum]
  \item The set of isotropy groups of a $G$-CW-complex;

  \item The family~$\TR$ consisting only of the trivial subgroup;
    
  \item The family~$\FIN$ consisting of all finite subgroups;
    
  \item The family~$\AME$ consisting of all amenable subgroups;

  \item\label{item:restricted family} Let $H$ be a subgroup of $G$ and let $\F$ be a family of
    subgroups of $G$.  Then the set $\F|_H=\{L\subset H\mid L\in\F\}$
    is a family of subgroups of $H$.  
  \end{enumerate}
\end{ex}

\begin{ex}[Families generated by a set of subgroups]\label{ex:family generated}
  Let $G$ be a group and let~$\calG$ be a non-empty set of subgroups.  The
  \emph{intersection-closed family~$\F\spann{\calG}$ generated
    by~$\calG$} is defined to be the smallest (with respect to
  inclusion) intersection-closed family containing $\calG$, that is
  \[
  \F\spann{\calG} = \biggl\{\bigcap_{i=1}^n g_iH_ig_i^{-1}
  \biggm| n \in \N, H_i\in \calG,g_i\in G \biggr\} \,.
  \]
  For~$n\in\IN$ we define the (conjugation-closed) family
  \[
  \F_n\spann{\calG} \coloneqq \biggl\{\bigcap_{i=1}^n g_iH_ig_i^{-1}
  \biggm| H_i\in \calG,g_i\in G\biggr\}\,.
  \]
\end{ex}

Recall that $EG$, the universal covering of an Eilenberg--MacLane
space~$BG$, is a terminal object in the $G$-homotopy category of free
$G$-CW-complexes.  The following is a generalisation to
$G$-CW-complexes with not necessarily trivial isotropy groups.

\begin{defn}[Classifying space for a family of subgroups]
  Let $G$ be a group and let $\F$ be a (conjugation-closed) family of subgroups of~$G$.
  A \emph{classifying space for~$G$ with respect to~$\F$} is
  a $G$-CW-complex~$E$ with the following universal property:
  \begin{itemize}
  \item All isotropy groups of~$E$ lie in~$\F$;
  \item For each $G$-CW-complex~$Y$ whose isotropy groups all lie in~$\F$,
    there is up to $G$-homotopy exactly one $G$-map~$Y \to E$.
  \end{itemize}
  We usually denote such classifying spaces by~$\EFG$ (even
  though they are only unique up to $G$-homotopy equivalence).
\end{defn}

  If $\F$ contains the trivial group, then the universal property
  of~$\EFG$ ensures that there exists a $G$-map $EG \to E_{\F}G$, which
  is unique up to $G$-homotopy.

\begin{thm}[Existence of classifying spaces for families of subgroups]
  \label{thm:EFGex}
  Let $G$ be a group and let $\F$ be a (conjugation-closed) family of
  subgroups.
  \begin{enumerate}[label=\enum]
  \item\label{item:fixed point sets}
    A $G$-CW-complex~$Y$ is a classifying space for~$G$ with respect
    to~$\F$ if and only if the following conditions are satisfied:
    \begin{itemize}
    \item All isotropy groups of~$Y$ lie in~$\F$;
    \item For all~$H \in \F$, the fixed point set~$Y^H$ is contractible.
    \end{itemize}
  \item\label{item:existence}
    There exists a classifying space for~$G$ with respect to~$\F$.
  \end{enumerate}
\end{thm}
\begin{proof}
  \ref{item:fixed point sets} This follows from an equivariant version of the Whitehead
  theorem~\cite[Theorem~1.6]{Lueck05} applied to the map~$Y \to G/G$.
  
  \ref{item:existence} In view of the first part, such a classifying space can be
  constructed by inductively attaching cells to kill homotopy groups
  of the fixed point sets~\cite[Proposition~2.3]{lueck_TG}. 
\end{proof}

\begin{rem}
  We point out that the construction of classifying spaces in
  Theorem~\ref{thm:EFGex}~\ref{item:existence} indeed works for (conjugation-closed)
  families of subgroups with no additional closure properties.  This
  level of generality is usually not considered in the literature,
  where classifying spaces are often defined only for families that
  are intersection-closed or closed under taking arbitrary subgroups.
\end{rem}

Many interesting constructions of classifying spaces for
intersection-closed families, most notably for $\FIN$, are
known~\cite[Section 4]{Lueck05}.  For a (conjugation-closed)
family~$\calG$ of subgroups of~$G$, we can consider the
intersection-closed family~$\F\spann{\calG}$ generated
by~$\calG$ (Example~\ref{ex:family generated}).  Then a model
for~$E_\calG G$ is given by the $G$-CW-subcomplex
of~$E_{\F\spann{\calG}}G$ consisting of all cells with isotropy
in~$\calG$.  
\begin{ex}
Let~$D_\infty=\spann{s,t\mid
  s^2=t^2=e}$ be the infinite dihedral group. A model for~$E_\FIN
D_\infty$ is given by the real line~$\IR$ on which $s$ and~$t$ act via
reflection at~$0,1\in \IR$, respectively. Considering the
(conjugation-closed) family~$\FIN\setminus\TR$, a model
for~$E_{\FIN\setminus\TR}D_\infty$ is given by the
subcomplex~$\IZ\subset \IR$, that is~$D_\infty/\spann{s}\sqcup
D_\infty/\spann{t}$.
\end{ex}

\subsection{$(G,\calH)$-CW-pairs}\label{subsec:G-CW pairs}

In this section we introduce a notion of pairs of equivariant
CW-complexes adapted to a collection of subgroups. 

\begin{defn}[Group pair]\label{defn:group pair}
  A \emph{group pair} is a pair~$(G,\calH)$, consisting
  of a group~$G$ and a collection~$\calH$ of
  subgroups of~$G$ (see Section~\ref{subsec:convs} for
  the term ``collection''). 
\end{defn}

For our purposes, the most important examples of group pairs will
arise as pairs of fundamental groups.

\begin{ex}[Fundamental group pair]\label{ex:fundamental group pair}
  Let $(X,A)$ be a CW-pair, where $X$ is
  path-connected, and let $x_0 \in X$. Moreover, we assume that each
  connected component of~$A$ is $\pi_1$-injective in~$X$.  A group
  pair~$(G,\calH)$ is a \emph{fundamental group pair} of~$(X,A)$ (at
  the basepoint~$x_0$) if:
  \begin{itemize}
  \item
    $G = \pi_1(X,x_0)$ and
  \item
    $\calH = (H_i)_{i \in I}$, where $A = \coprod_{i \in I} A_i$ is a
    decomposition of~$A$ into connected components and for each~$i \in
    I$, there exists a basepoint~$x_i \in A_i$ and a path~$\gamma_i$ in~$X$
    from~$x_0$ to~$x_i$, such that $H_i$ is the subgroup
    of~$\pi_1(X,x_0)$ isomorphic to~$\pi_1(A_i,x_i)$ via the
    homomorphism induced by~$\gamma_i$.
  \end{itemize}
  
  In this situation, we will also abuse notation and just
  write~$\pi_1(A_i)$ for this subgroup~$H_i$ in~$G$. It should be
  noted that there is always an implicit (fixed) choice of basepoints
  and paths involved. In many cases, these choices will be of no
  consequence; when these choices would matter, we will be in
  situations, where we include all $G$-conjugates of the~$(H_i)_{i\in I}$ in the
  collection of subgroups in question, and thus avoid ambiguities.
\end{ex}

\begin{defn}[$(G,\calH)$-CW-pair]\label{def:GHpair}
  Let $(G, \calH)$ be a group pair with~$\calH = (H_i)_{i \in I}$.
    A \emph{$(G,\calH)$-CW-pair} is a $G$-CW-pair~$(Y,B)$ together with
    a decomposition $$B=\coprod_{i\in I} G\times_{H_i} B_i \,,$$ where $B_i$ is
    an $H_i$-CW-complex.

    Let $\F$ be a family of subgroups of~$G$. 
    We say that $(Y,B)$ \emph{has isotropy in $\F$} if all
    isotropy groups of~$Y$ lie in~$\F$.

    The \emph{relative dimension}~$\dim(Y,B)\in \IN\cup\{\infty\}$ is the
    dimension of the relative $G$-CW-complex $(Y,B)$.

    A map of $(G,\calH)$-CW-pairs $f\colon (Y,B)\to (Z,C)$ is a $G$-map
    of pairs such that the restriction~$f|_B$ is of the form~$\coprod_{i\in I}
    G\times_{H_i} f_i$, where $f_i\colon B_i\to C_i$ is an $H_i$-map.
\end{defn}

If a~$(G,\calH)$-CW-pair~$(Y,B)$ as above has isotropy in~$\F$, then
the isotropy groups of the~$H_i$-CW-complex~$B_i$ lie in~$\F|_{H_i}$.

In the situation of Definition~\ref{def:GHpair}, replacing a
subgroup~$H_i$ by a conjugate~$gH_ig^{-1}$ with~$g \in G$ corresponds
to changing the reference point in the description of the induced
space~$G \times_{H_i} B_i$.  Thus, the
collection~$\calH$ can also be viewed as a collection of representatives
of conjugacy classes of subgroups of~$G$.

\begin{ex}[Universal covering pair]\label{ex:universal covering pair}
  Let $(X,A)$ be a CW-pair with fundamental group
  pair~$(G,\calH)$, where $A = \coprod_{i \in I} A_i$ and $\calH =
  (H_i)_{i \in I} = (\pi_1(A_i))_{i \in I}$ are as in
  Example~\ref{ex:fundamental group pair}.  Denote by $p\colon
  \widetilde{X}\to X$ the universal covering map.  Then there is a
  $G$-homeomorphism~$p^{-1}(A)\cong \coprod_{i\in I} G\times_{H_i}
  \widetilde{A_i}$, where $\widetilde{A_i}$ is the universal covering
  of $A_i$.  This shows that $(\widetilde{X},p^{-1}(A))$ is a
  $(G,\calH)$-CW-pair with isotropy in the trivial family~$\TR$.
\end{ex}

\subsection{Classifying spaces of group pairs with respect to a family}

We now let families of subgroups and an additional collection of
subgroups interact:

\begin{defn}[Classifying space for group pairs]\label{defn:EFGH}
  Let $G$ be a group, $\calH$ be a collection of subgroups of $G$, and
  $\F$ be a family of subgroups of $G$. A \emph{classifying space for
    the group pair~$(G,\calH)$ with respect to~$\F$} is a
  $(G,\calH)$-CW-pair $(E,D)$ with the following universal property:
  \begin{itemize}
  \item The pair $(E,D)$ has isotropy in~$\F$;
  \item For each $(G,\calH)$-CW-pair~$(Y,B)$ with isotropy in~$\F$,
    there is up to $G$-homotopy exactly one $G$-map~$(Y,B)\to (E,D)$
    of $(G,\calH)$-CW-pairs.
  \end{itemize} 
  We usually denote such classifying spaces by~$\EFGH$ (even though
  they are only unique up to $G$-homotopy equivalence of pairs).
\end{defn}

Models for~$\EFGH$ can be constructed as mapping cylinders:

\begin{lem}[Existence of classifying spaces for group pairs]\label{lem:existence classifying pair}
	Let~$(G,\calH)$ be a group pair and let~$\F$ be family of subgroups 
	of~$G$. Then there exists a classifying space for the 
	group pair~$(G,\calH)$ with respect to~$\F$.
\end{lem}
\begin{proof}
	We write the collection~$\calH$ as~$(H_i)_{i\in I}$.
  For every~$i\in I$, let $E_{\F|_{H_i}}H_i$ be a classifying space
  for~$H_i$ with respect to the family~$\F|_{H_i}$.  The induced
  $G$-CW-complex~$G\times_{H_i}(E_{\F|_{H_i}}H_i)$ has isotropy
  in the family~$\F$ and  
  hence there exists a 
  ~$G$-map~$G\times_{H_i}(E_{\F|_{H_i}}H_i)\to \EFG$
  (that is unique up to $G$-homotopy).
  Then the mapping cylinder of the $G$-map $$\coprod_{i\in
    I}G\times_{H_i}(E_{\F|_{H_i}}H_i)\to \EFG$$ is a model
  for~$\EFGH$. This follows from the universal properties of the
  classifying spaces~$E_{\F|_{H_i}}H_i$ and~$\EFG$.
\end{proof}

\section{Relative cohomological dimension and mapping degrees}

We discuss an application of classifying spaces for group pairs and
relative group cohomology to maps between manifolds.  We obtain
inheritance results for the freeness of fundamental groups of
manifolds in terms of mapping degrees.  The results of this section
are independent from the rest of the paper.

\subsection{Relative group cohomology}

We recall the definition of relative cohomology of a group pair and a
characterisation of groups pairs with relative cohomological dimension
one.

\begin{defn}[Relative cohomology of group pairs]\label{defn:relative group cohomology}
  Let $(G,\calH)$ be a group pair
  and $R$ be a commutative ring. We define the \emph{relative cohomology}
  $H^*(G,\calH;V)$ with coefficients in an $RG$-module $V$ to be the
  $R$-module
  \[
  H^*(G,\calH;V)\coloneqq H^*_G(E_{\TR}(G,\calH);V) \,,
  \]
  where $E_{\TR}(G,\calH)$ is a classifying space for $(G,\calH)$ with
  respect to the trivial family.  
  Here
  the equivariant cohomology~$H^*_G(E_\TR(G,\calH);V)$ is by
  definition the cohomology~$H^*(BG,\coprod_{H\in \calH}BH;V)$ with
  twisted coefficients.
\end{defn}

\begin{defn}[Relative cohomological dimension]\label{defn:rel:cohom:dim}
  The \emph{relative cohomological dimension} $\cd_R(G,\calH)$ of the
  group pair $(G,\calH)$ over the ring $R$ is defined as follows:
  $$
  \cd_R(G,\calH) \coloneqq \sup\{n \in \N\mid H^n(G,\calH;V) \not \cong 0 \mbox{ for some $RG$-module }V\}.
  $$  
  For simplicity we denote~$\cd_\IZ(G,\calH)$ by~$\cd(G,\calH)$.
\end{defn}

To illustrate these definitions, we mention that
for $\calH= (H_i)_{i \in I}$ the long exact sequence for the pair
$E_{\TR}(G,\calH)$ takes the following form:
\[
	\cdots\to \prod_{i\in I}H^{n-1}(H_i;V)\to H^n(G,\calH;V)\to H^n(G;V)\to \prod_{i\in I}H^n(H_i;V)\to \cdots \,.
\]
This shows that $\cd_R(G,\calH)\le n$ if and only if for all
$RG$-modules $V$ the restriction map $H^k(G;V)\to \prod_{i\in
  I}H^k(H_i;V)$ is an isomorphism for $k>n$ and an epimorphism for
$k=n$.

\begin{rem}
  Let $(G,\calH)$ be a group pair with~$\calH = (H_i)_{i \in I}$.
  While our definition of~$H^*(G,\calH;V)$ is purely topological, one
  may also define it algebraically via derived functors. More
  precisely, consider the augmentation $RG$-map~$R[\coprod_{i\in
      I}G/H_i]\to R$ and let $\Delta$ denote its kernel. Then there
  exists a natural isomorphism
  \[
  H^*(G,\calH;V)\cong \Ext^{*-1}_{RG}(\Delta,V) \,.
  \]
  In this situation we have that
  $\cd_R(G,\calH)=\operatorname{pd}_{RG}(\Delta)+1$, where
  $\operatorname{pd}_{RG}$ denotes the projective dimension.

  Many of the usual cohomological tools for group cohomology
  have been developed for the relative case as
  well~\cite{Takasu59,Bieri-Eckmann78,Alonso91}.
  
  In the case of a single subgroup~$H$, the relation between~$H^*(G,H;V)$
  and the Bredon cohomology~$H^*_G(E_{\F\spann{H}}G;V)$ has recently
  been investigated~\cite{ANCM17,ANCMSS21}.
\end{rem}

By the work of Stallings~\cite{Stallings} and Swan~\cite{Swan} groups
of cohomological dimension one are precisely the free groups.  
The following is a generalisation to the relative setting.

\begin{thm}[{Group pairs of relative cohomological dimension one~\cite{Dicks80,Alonso91}}]
\label{thm:relative Stallings-Swan}
  Let $(G,\calH)$ be a group pair with~$\calH = (H_i)_{i \in I}$.
  Then the following are equivalent:
  \begin{enumerate}[label=\enum]
  	\item $\cd(G,\calH)=1$;
	\item There exists a free group~$F$
  such that~$G\cong F\ast \Asterisk_{i\in I}H_i$.
  \end{enumerate}
\end{thm}

\subsection{Mapping degrees and monotonicity}

For maps between manifolds with $\pi_1$-injective boundary components,
we prove monotonicity results on the cohomological dimension of the
fundamental group pairs.

First, we introduce some notation.  Let $(X,A)$ be a CW-pair with
$X$ path-connected.
Let~$A=\coprod_{i\in I}A_i$ be a decomposition into connected components
and assume that each~$A_i$ is $\pi_1$-injective in~$X$. We
denote by $(\pi_1(X),\pi_1(A))$ a fundamental group pair of~$(X,A)$
(Example~\ref{ex:fundamental group pair}) and by~$p\colon
\widetilde{X}\to X$ the universal covering map.  
Let $H^*(X,A;V)$ be the cohomology with twisted coefficients in a
$\pi_1(X)$-module~$V$. Then the classifying map~$\varphi_{(X,A)}\colon
(\widetilde{X},p^{-1}(A))\to E_{\TR}(\pi_1(X),\pi_1(A))$ induces a map
on cohomology:
\[
	H^*(\varphi_{(X,A)})\colon H^*(\pi_1(X),\pi_1(A);V)\to H^*(X,A;V) \,.
\]

\begin{lem}\label{lemma:varphi:2:injective}
  Let $(X,A)$ be a CW-pair with fundamental group
  pair~$(\pi_1(X),\pi_1(A))$ as above.  The map $H^2(\varphi_{(X,A)}) \colon
  H^2(\pi_1(X),\pi_1(A);V)\to H^2(X,A;V)$ is injective for every
  $\pi_1(X)$-module~$V$.
\end{lem}
\begin{proof}
We argue via the four lemma for monomorphisms. We consider the following commutative diagram, where the coefficient module is omitted:
\[
\xymatrix{
H^1(\pi_1(X)) \ar[r] \ar[d] & \prod_{i\in I}H^1(\pi_1(A_i)) \ar[r] \ar[d] & H^2(\pi_1(X), \pi_1(A))  \ar[r] \ar[d]^-{H^2(\varphi_{(X,A)})}  & H^2(\pi_1(X)) \ar[d] \\
H^1(X) \ar[r] & H^1(A) \ar[r] & H^2(X, A) \ar[r] & H^2(X) \,.
}
\]
Here all vertical maps are induced by the respective classifying maps.
Since a model for~$E_\TR(\pi_1(X))$ [resp.~for $E_\TR(\pi_1(A_i))$]
can be built from~$\widetilde{X}$ [resp.~from~$\widetilde{A_i}$] 
by attaching cells of dimension greater than or equal to~$3$,
the first and second vertical arrows
are isomorphisms, while the last vertical arrow is injective. 
Applying the four lemma for monomorphisms, we conclude that $H^2(\varphi_{(X,A)})$ is injective.
\end{proof}

We use the convention (Section~\ref{subsec:convs}) to say that a
manifold~$M$ has \emph{$\pi_1$-injective boundary}~$\partial M$ if every
component of~$\partial M$ is $\pi_1$-injective in $M$.

\begin{thm}[Mapping degree and relative cohomological dimension one]\label{thm:degree cd one}
  Let $f\colon (M,\partial M)\to (N,\partial N)$ be a map between
  oriented compact connected manifolds of the same dimension with
  $\pi_1$-injective boundary.  Then the following hold:
  \begin{enumerate}[label=\enum]
  \item If $\deg(f) =\pm 1$ and $\cd(\pi_1(M),\pi_1(\partial M))\le
    1$, then $\cd(\pi_1(N),\pi_1(\partial N))\le 1$;
  \item If $\deg(f)\neq 0$ and $\cd_\IQ(\pi_1(M),\pi_1(\partial M))\le
    1$, then $\cd_\IQ(\pi_1(N),\pi_1(\partial N))\le 1$.
  \end{enumerate}
\end{thm}

\begin{proof}
  We proceed by contraposition.  Let $R=\IZ$ [resp.~$R=\IQ$] and
  suppose that $\cd_R(\pi_1(N),\pi_1(\partial N))>1$. Then by a
  dimension shifting argument, there exists an $R\pi_1(N)$-module $V$
  such that $H^2(\pi_1(N),\pi_1(\partial N);V)$ is non-trivial. We
  denote by $f^{-1}V$ the $R\pi_1(M)$-module that is obtained from $V$
  by restriction along $\pi_1(f)$. Consider the following commutative
  diagram:
  \[\begin{tikzcd}
  H^2(N,\partial N;V)\ar{rr}{H^2(f)} && H^2(M,\partial M;f^{-1}V) \\
  H^2(\pi_1(N),\pi_1(\partial N);V)\ar{u}{H^2(\varphi_{(N,\partial N)})}\ar{rr}{H^2(\pi_1(f))} && H^2(\pi_1(M),\pi_1(\partial M);f^{-1}V)\ar{u}{H^2(\varphi_{(M,\partial M)})} \,.
  \end{tikzcd}
  \]
  Here the vertical maps, which are induced by the respective
  classifying maps, are injective in degree~$2$
  (Lemma~\ref{lemma:varphi:2:injective}).  By Poincar\'e--Lefschetz
  duality with twisted coefficients, there exists an Umkehr map
  \[
  f_!\colon H^2(M,\partial M;f^{-1}V)\to H^2(N,\partial N;V) 
  \]
  such that the composition $f_!\circ H^2(f) \colon H^2(N,\partial
  N;V)\to H^2(N,\partial N;V)$ is given by multiplication
  with~$\deg(f)$.  Hence the map~$H^2(f)$ is injective in each of the following
  cases:
  \begin{enumerate}[label=\enum]
  \item\label{item:rel cd1} If $\deg(f)= \pm
    1$ and $R=\IZ$;
    
  \item\label{item:rel cd2} If $\deg(f)\neq 0$ and $R=\IQ$.
  \end{enumerate}  
  This shows that in the situations~\ref{item:rel cd1} and~\ref{item:rel cd2} the composition
  $$H^2(f)\circ H^2(\varphi_{(N,\partial N)})=H^2(\varphi_{(M,\partial M)})\circ H^2(\pi_1(f))$$ is
  injective, whence $H^2(\pi_1(f))$ is injective.  Therefore the
  relative cohomology group $H^2(\pi_1(M),\pi_1(\partial M);f^{-1}V)$ is
  non-trivial and we have $\cd_R(\pi_1(M),\pi_1(\partial M)) > 1$.
\end{proof}

\begin{rem}
If the universal coverings of $N$ and $\partial N$ are
$k$-connected, then the map $H^{k+1}(\varphi_{(N,\partial N)})$ is injective and
hence similar monotonicity results hold for cohomological dimension
at most~$k$.
\end{rem}

Theorem~\ref{thm:degree cd one} readily implies the following:

\begin{cor}\label{cor:degree one relative}
  Let $f\colon (M,\partial M)\to (N,\partial N)$ be a map between
  oriented compact connected manifolds of the same dimension with
  $\pi_1$-injective boundary. Suppose that~$\partial M=\coprod_{i=1}^m M_i$
  and $\partial N=\coprod_{i=1}^n N_i$ are decompositions into
  connected components.  If $\deg(f)=\pm 1$ and there exists a free
  group~$F_M$ such that $$\pi_1(M)\cong F_M\ast \Asterisk_{i=1}^m
  \pi_1(M_i)\,,$$ then there exists a free group~$F_N$ such
  that~$\pi_1(N)\cong F_N\ast \Asterisk_{i=1}^n \pi_1(N_i)$.
\end{cor}        
\begin{proof}
  This follows from Theorem~\ref{thm:degree cd one}~(i) and the
  group-theoretic characterisation of relative cohomological dimension
  one (Theorem~\ref{thm:relative Stallings-Swan}).
\end{proof}

Examples of manifolds satisfying the assumptions of
Corollary~\ref{cor:degree one relative} are the following:

\begin{ex}
  Let $F_k$ be a free group of rank~$k$, and let $\calH=(H_1, \dots,
  H_m)$ be a finite collection of finitely presented groups. Then for
  every~$n\ge 7$, there exists a compact connected $n$-dimensional
  manifold~$(M,\partial M)$ with fundamental group pair~$(G,\calH)$
  such that~$G\cong F_k\ast \Asterisk_{i=1}^m H_i$. Indeed, let $L_i$
  be an oriented closed connected $4$-manifold with~$\pi_1(L_i) \cong
  H_i$ and consider
  \[
  (M,\partial M)\coloneqq(\#_{i=1}^k S^1\times S^{n-1}) \mathbin{\#} (\#_{i=1}^m L_i\times D^{n-4}) \,.
  \]
  Then $\partial M \cong \coprod_{i=1}^m L_i\times S^{n-5}$ and
  $\pi_1(M) \cong F_k\ast \Asterisk_{i=1}^m \pi_1(L_i)$.
\end{ex}

In the case of closed manifolds, Corollary~\ref{cor:degree one
  relative} provides a simplified proof of a result by Dranishnikov
and Rudyak~\cite[Theorem 5.2]{Dranishnikov-Rudyak09}, without making
use of the Berstein class.  We also obtain the following analogue for
maps of non-zero degree:

\begin{cor}\label{cor:finmono}
  Let $f\colon M\to N$ be a map between oriented closed connected
  manifolds of the same dimension with~$\deg(f)\neq 0$. If $\pi_1(M)$
  is the fundamental group of a graph of finite groups, then so
  is~$\pi_1(N)$.  In particular, if $\pi_1(M)$ is virtually free, then
  $\pi_1(N)$ is virtually free.
\end{cor}
\begin{proof}
  This follows from Theorem~\ref{thm:degree cd one}~(ii) and
  Dunwoody's characterisation~\cite{Dunwoody79} of groups of
  cohomological dimension one over arbitrary rings. 
  As fundamental groups of closed manifolds are finitely
  generated, this characterisation can also be expressed
  in terms of virtual freeness~\cite[Corollary~1.2]{Dunwoody79}.
\end{proof}

We are not aware of a relative version of Dunwoody's result that would
characterise group pairs of relative cohomological dimension one over
arbitrary rings. The results in this section motivate the following:

\begin{question}\label{q:Ccat1}
  For which other classes~$\calC$ of groups does the following hold?
  Whenever $f \colon M \to N$ is a map between oriented closed
  connected manifolds of the same dimension with~$\deg(f) \neq 0$ [or~$\deg(f)=1$] and
  $\pi_1(M)$ is the fundamental group of a graph of groups
  from~$\calC$, also $\pi_1(N)$ must be the fundamental group of a
  graph of groups from~$\calC$. 
\end{question}

Positive answers to Question~\ref{q:Ccat1} lead to corresponding
monotonicity results for the generalised LS-category~$\leq 2$
associated with the class~$\calC$, provided that $\calC$ is closed
under isomorphisms, subgroups, and quotients~\cite[Corollary~5.4 and
  the following paragraph]{CLM20}.

\section{Relative open covers and equivariant nerve pairs}

In this section, we study equivariant nerves of open covers in a
relative setting.  Given an open cover of a space, the nerve of the
cover gives an approximation of the space. Considering actions on
spaces and compatible covers leads to equivariant nerves, studied by
L\"oh and Sauer~\cite{Loeh-Sauer19} for universal covering spaces with
the action by the fundamental group. We adapt this approach to pairs
of spaces.

\subsection{Open covers and nerve pairs}

We fix some notation and terminology on open covers and their nerves. 

Let $Y$ be a space and $\calV$ be a cover of~$Y$ by path-connected
open subsets. We regard $\calV$ as a set of subsets of~$Y$ (and not
as a collection of subsets).

\begin{defn}[$\F$-Cover]
  Let $\F$ be a family of subgroups of~$\pi_1(Y)$. We say that~$\calV$
  is an \emph{$\F$-cover} of~$Y$ if $\im(\pi_1(V,x)\to \pi_1(Y,x))\in
  \F$ for all~$V\in \calV$ and all~$x \in V$.
\end{defn}

\begin{defn}[Convex cover]
  The cover~$\calV$ of~$Y$ is said to be \emph{convex} if every
  intersection of finitely many elements of~$\calV$ is path-connected
  or empty.
\end{defn}

The \emph{multiplicity~$\mult(\calV)$} of~$\calV$ is defined as
follows:
$$
\mult(\calV) \coloneqq \sup\Bigl\{n \in \N \Bigm| \bigcap_{i = 1}^n V_i \neq \emptyset \mbox{ for some pairwise different } V_1, \dots, V_n \in \calV\Bigr\} \,.
$$

The \emph{nerve}~$N(\calV)$ of~$\calV$ is the (abstract) simplicial
complex with vertex set~$\calV$; pairwise different~$V_0, \dots, V_n
\in \calV$ span an $n$-simplex in~$N(\calV)$ if~$V_0 \cap \dots \cap V_n
\neq \emptyset$. By definition, $\dim(N(\calV))=\mult(\calV)-1$.

Let $|N(\calV)|$ be the geometric realisation of the nerve~$N(\calV)$.  Given a
partition of unity~$(\psi_V)_{V\in\calV}$ on~$Y$ subordinate
to~$\calV$, there is an associated \emph{nerve map}
\begin{equation}\label{eqn:nerve map}
  \mu\colon Y\to |N(\calV)|\,,
  \quad y\mapsto\mu(y)\coloneqq\sum_{V\in\calV}\psi_V(y)\cdot V \,.
\end{equation} 
The nerve map is unique up to homotopy, since different
choices of partitions of unity lead to homotopic nerve maps.

\begin{defn}[Relative multiplicity]\label{def:relative:mult}
  For a subspace~$B$ of~$Y$, we define the \emph{relative
    multiplicity}~$\mult_B(\calV)$ of~$\calV$ (with respect to~$B$) as
  follows:
  \begin{align*}
    \mult_B(\calV) \coloneqq \sup\Bigl\{n \in \N \Bigm| 
    &\bigcap_{i = 1}^n V_i\neq \emptyset \mbox{ and } B \cap \Bigl(\bigcap_{i = 1}^n V_i\Bigr)=\emptyset \\
    &\mbox{for some pairwise different } V_1, \ldots, V_n \in \calV\Bigr\} \,.
  \end{align*}
\end{defn}

\begin{defn}[Nerve pair]\label{defn:nerve pair}
  For a subspace~$B$ of~$Y$, we denote by~$N_B(\calV)$ the simplicial
  subcomplex of~$N(\calV)$ with vertex set $\calV_B\coloneqq\{V\in
  \calV\mid V\cap B\neq \emptyset\}$, and pairwise
  different~$V_0,\dots, V_n\in \calV_B$ span a simplex in~$N_B(\calV)$
  if $V_0\cap \dots\cap V_n\cap B\neq \emptyset$. By
  construction we have $N(\calV)=N_Y(\calV)$.
  
  The nerve map~$\mu\colon Y\to |N(\calV)|$ induces a map of pairs
  \[
  \mu\colon (Y,B)\to (|N(\calV)|,|N_B(\calV)|) \,.
  \]

  The \emph{relative dimension}~$\dim(N(\calV),N_B(\calV))$ is the
  dimension of the relative simplicial
  complex~$(N(\calV),N_B(\calV))$.  By definition,
  $\dim(N(\calV),N_B(\calV))=\mult_B(\calV)-1$.
\end{defn}

\subsection{Equivariant nerve pairs}

We now consider open covers that are compatible with a group action,
giving rise to an action on their nerve.

\begin{defn}[Invariant cover and partition of unity]\label{defn:covers}
  Let $Y$ be a $G$-CW-complex and $\calV$ be a cover of~$Y$ by
  path-connected open subsets. We say that $\calV$ is
  \emph{$G$-invariant} if for all~$g\in G$ and $V\in \calV$, we have
  $g\cdot V\in \calV$.  For~$V\in \calV$, we write
  \[
  \Stab_G(V)\coloneqq \{g\in G\ |\ g\cdot V=V\} \,.
  \]
  For a family of subgroups~$\F$ of~$G$, we say that the cover~$\calV$
  \emph{has isotropy in $\F$} if for all~$V\in \calV$, we have
  $\Stab_G(V)\in \F$.
	
  A partition of unity~$(\psi_V)_{V\in\calV}$ on~$Y$ subordinate
  to a $G$-invariant cover~$\calV$ is said to be \emph{$G$-invariant} if for all $g\in G$ and
  $y\in Y$, we have
  \[
  \psi_V(y)=\psi_{g\cdot V}(g\cdot y) \,.
  \]
\end{defn}

The key examples will come from covers of a space~$X$ giving rise
to~$\pi_1(X)$-invariant covers of the universal covering
space~$\widetilde{X}$ (Example~\ref{ex:covers}).

We recall basic properties of equivariant nerves~\cite[Lemmas~4.8
  and~4.11]{Loeh-Sauer19} and their proofs for completeness.

\begin{lem}[Equivariant nerve]\label{lem:nerve}
  Let $Y$ be a $G$-CW-complex and $\calV$ be a $G$-invariant cover
  of~$Y$.  Then the following hold:
  \begin{enumerate}[label=\enum]
  \item\label{item:nerve G-CW} Let~$B$ be an $H$-invariant subcomplex of~$Y$ for a subgroup~$H$ of~$G$. Then $N_B(\calV)$ is an $H$-simplicial complex
    and its geometric realisation~$|N_B(\calV)|$
    is an $H$-CW-complex. In particular, $N(\calV)$
    is a $G$-simplicial complex;
  \item\label{item:moving sets} 
    Suppose that $g\cdot V\cap V\neq \emptyset$ implies~$g\cdot V=V$
    for all~$g\in G,V\in \calV$.  Let $\F$ be an intersection-closed
    family of subgroups of~$G$.  If the cover~$\calV$ has isotropy
    in~$\F$, then the $G$-CW-complex~$|N(\calV)|$ has isotropy in~$\F$;
  \item\label{item:G-nerve map} Let $(\psi_V)_{V\in\calV}$ be a $G$-invariant partition of
    unity on~$Y$ subordinate to~$\calV$. Then the induced nerve
    map~$\mu\colon Y\to |N(\calV)|$ is $G$-equivariant.
  \end{enumerate}
\end{lem}
\begin{proof}
  \ref{item:nerve G-CW} To show that $N_B(\calV)$ is an $H$-simplicial complex, it
  suffices to prove that the $H$-action sends simplices of
  $N_B(\calV)$ to simplices of $N_B(\calV)$.  Let~$v$ be a vertex
  of~$N_B(\calV)$ corresponding to~$V\in\calV$ with $V\cap B\neq
  \emptyset$. Then for every~$h\in H$, we have $\emptyset\neq h\cdot
  (V\cap B)=(h\cdot V)\cap (h\cdot B)\subset (h\cdot V)\cap B$.  This
  shows that the vertex~$h\cdot v$ of~$N(\calV)$ corresponding
  to~$h\cdot V\in \calV$ lies in~$N_B(\calV)$. The same argument also
  extends to higher-dimensional simplices, which proves the claim.  Then the
  geometric realisation~$|N_B(\calV)|$ is an $H$-CW-complex
  (Example~\ref{ex:from:sc:to:cw}).
		
  \ref{item:moving sets} We show that the isotropy groups of the vertices of the
  barycentric subdivision of $N(\calV)$ lie in $\F$. This is indeed
  sufficient; since the action is simplicial the stabiliser of every
  interior point of a $k$-simplex in the barycentric subdivision of~
  $N(\calV)$ is the intersection of the stabilisers of its $k+1$
  vertices. Then the fact that~$\F$ is closed under finite
  intersections yields the thesis.
  
  Let $v$ be a vertex in the barycentric subdivision of $N(\calV)$,
  associated to a $k$-simplex corresponding to~$V_0,\ldots,V_k \in
  \calV$ with $V_0\cap\cdots\cap V_k\neq\emptyset$. It remains to show
  that the subgroup
  \[
  G_v = \bigl\{g\in G
  \bigm| \{g\cdot V_0,\ldots, g\cdot V_k\}=\{V_0,\ldots,V_k\}
  \bigr\}
  \]
  of~$G$ lies in the family~$\F$. For $k=0$,
  we know that $G_v=\Stab_G(V_0)\in \F$ by the assumption that $\calV$ has isotropy
  in $\F$. On the other hand, for $k>0$, by the assumption that
  $g\cdot V\cap V\neq\emptyset$ implies $g\cdot V=V$, we have that
  $g\cdot V_i=V_j$ for~$i,j\in \{0,\ldots,k\}$ implies $i=j$. Hence
  $G_v=\Stab_G(V_0)\cap\cdots\cap \Stab_G(V_k)$. This group lies
  in~$\F$ since $\F$ is closed under finite intersections.
		
  \ref{item:G-nerve map} The $G$-invariance of the partition of unity implies the
  $G$-equivariance of the nerve map~\eqref{eqn:nerve map} as follows:
  For all~$g\in G$ and all~$y\in Y$, we have
  \begin{align*}
    \mu(g\cdot y) &= \sum_{V\in\calV}\psi_V(g\cdot y)\cdot V 
    = \sum_{V\in\calV}\psi_{g^{-1}\cdot V}(y)\cdot V \\
    &= \sum_{V\in\calV}\psi_{V}(y)\cdot (g\cdot V)
    = g\cdot \mu(y) \,.
  \end{align*}
  This finishes the proof.
\end{proof}

We extend the previous results to the relative situation:

\begin{lem}[Equivariant nerve pair]\label{lem:nerve pair}
  Let $(G,\calH)$ be a group pair and let
  $(Y,B)$ be a~$(G,\calH)$-CW-pair.  Let $\F$ be an
  intersection-closed family of subgroups of~$G$ and $\calV$ be
  a~$G$-invariant cover of~$Y$ with isotropy in~$\F$.  Suppose that
  the following hold:
  \begin{enumerate}[label=\enum]
  \item\label{item:nerve pair1} For all~$V\in\calV,g\in G$ with~$g\cdot V\cap V\neq
    \emptyset$, we have~$g\cdot V=V$;
  \item\label{item:nerve pair2} There exists a $G$-invariant partition of unity on $Y$
    subordinate to $\calV$;
  \item\label{item:nerve pair3} For all~$V\in\calV$ with $V\cap B\neq \emptyset$, the
    intersection $V\cap B$ is connected.
  \end{enumerate}
  Then $(|N(\calV)|,|N_B(\calV)|)$ is a $(G,\calH)$-CW-pair with
  isotropy in~$\F$. Moreover, the nerve map~$\mu\colon Y\to |N(\calV)|$
  induces a map of $(G,\calH)$-CW-pairs:
  $$
  \mu\colon (Y,B)\to \bigl(|N(\calV)|,|N_B(\calV)|\bigr) \,.
  $$
\end{lem}
\begin{proof}
  By Lemma~\ref{lem:nerve}, assumption~\ref{item:nerve pair1} implies that
  ~$|N(\calV)|$ is a~$G$-CW-complex with
  isotropy in~$\F$, and assumption~\ref{item:nerve pair2} implies that the nerve map
  $\mu\colon Y\to |N(\calV)|$ is $G$-equivariant.
  
  We write the collection~$\calH$ as $(H_i)_{i\in I}$.
  Since $(Y,B)$ is a $(G,\calH)$-CW-pair, we have a decomposition
  $B= \coprod_{i\in I}G\times_{H_i} B_i$, where $B_i$ is an
  $H_i$-CW-complex. We identify~$B_i$ with the subset $[e,B_i]\subset
  B\subset Y$, where $e \in G$ denotes the neutral element. 
  We also identify the set of vertices of~$|N(\calV)|$ with~$\calV$.
  There is a~$G$-map
  \[
  	\Phi\colon \coprod_{i\in I}G\times_{H_i}|N_{B_i}(\calV)|\to |N_B(\calV)| \,,
  \]
  mapping a vertex~$[g,V]$ of~$G\times_{H_i} |N_{B_i}(\calV)|$ to 
  the vertex~$g\cdot V$ of~$|N_B(\calV)|$, and that is defined by affine extension.
  The affine extension is well-defined, 
  since the images of vertices spanning a simplex in~$|N_{B_i}(\calV)|$ 
  also span a simplex in~$|N_B(\calV)|$.
  We claim that assumption~\ref{item:nerve pair3} implies that~$\Phi$ is
  a~$G$-homeomorphism.
  
  Indeed, the inverse map~$\Phi^{-1}$ is given as follows:
  For a vertex~$V$ of~$|N_B(\calV)|$, we have $V\cap B\neq\emptyset$ and
  hence the intersection~$V\cap B$ is connected by assumption~\ref{item:nerve pair3}.
  Thus, there exists a unique element~$i\in I$ and a unique coset~$gH_i\in G/H_i$
  such that $V\cap [g,B_i]\neq\emptyset$. 
  Since $V\cap[g,B_i]=g(g^{-1}\cdot V\cap B_i)$, 
  we may define~$\Phi^{-1}$ to map the vertex~$V$ to 
  the vertex~$[g,g^{-1}\cdot V]$.
  This assignment is~$G$-equivariant: Indeed, for every~$g' \in G$, we have
  $\emptyset\neq g'(V\cap [g,B_i])=g' g(g^{-1}\cdot V\cap B_i)$. 
  Hence under~$\Phi^{-1}$ the vertex~$g' \cdot V$ is mapped to
  the vertex~$[g' g,g^{-1}\cdot V]$. 
  
  Then~$\Phi^{-1}$ is determined by affine extension. This is well-defined
  because the images of vertices spanning a simplex in~$|N_B(\calV)|$
  also span a simplex in the corresponding~$|N_{B_i}(\calV)|$.
  Thus~$\Phi$ is a~$G$-homeomorphism, showing that 
  $(|N(\calV)|,|N_B(\calV)|)$ is a~$(G,\calH)$-CW-pair.
   
  The $G$-map~$\mu\colon (Y,B)\to (|N(\calV)|,|N_B(\calV)|)$ is 
  a map of~$(G,\calH)$-CW-pairs, since we have $\mu(B_i)\subset |N_{B_i}(\calV)|$.
\end{proof}

\subsection{Relative open covers}
\label{sec:lifted covers}

We study our main example of equivariant nerve pairs coming from
lifted covers of CW-pairs.

\begin{ex}[Lifted cover]\label{ex:covers}
  Let $X$ be a connected CW-complex, let
  $G\coloneqq\pi_1(X)$, and let $p\colon \widetilde{X}\to X$ denote the
  universal covering.  For a cover~$\calU$ of~$X$ by path-connected
  open subsets, we consider the \emph{lifted
    cover}~$\widetilde{\calU}$ of~$\widetilde{X}$:
  \[
  \widetilde{\calU}\coloneqq
  \bigl\{V\subset \widetilde{X}
  \bigm| V \text{ is a path-connected component of } p^{-1}(U) \text{ for some } U\in \calU
  \bigr\} \,.
  \]
  Clearly, $\widetilde{\calU}$ is a $G$-invariant cover of
  $\widetilde{X}$.  Note that for every~$g\in G$,
  $V\in\widetilde{\calU}$, the condition $g\cdot V\cap V\neq
  \emptyset$ implies~$g\cdot V=V$.  Moreover, for
  every~$V\in\widetilde{\calU}$ we have that
  $\Stab_G(V)$ is conjugate to $\im(\pi_1(p(V))\to \pi_1(X))$.
  This shows that if $\calU$ is an~$\F$-cover of $X$, then
  $\widetilde{\calU}$ has isotropy in $\F$.
	
  Every given partition of unity $(\varphi_U)_{U\in\calU}$ on $X$
  subordinate to $\calU$ lifts to a $G$-invariant partition
  of unity $(\widetilde{\varphi}_V)_{V\in\widetilde{\calU}}$ on
  $\widetilde{X}$ subordinate to $\widetilde{\calU}$ as follows: For
  $V\in \widetilde{\calU}$, define
  \[
  \widetilde{\varphi}_V\coloneqq \chi_V\cdot (\varphi_{p(V)}\circ p)\colon \widetilde{X}\to [0,1] \,,
  \]
  where $\chi_V\colon \widetilde{X}\to [0,1]$ denotes the
  characteristic function on $V\subset \widetilde{X}$.  Let $\nu$
  and~$\widetilde{\nu}$ be the nerve maps associated
  to~$(\varphi_U)_{U\in\calU}$
  and~$(\widetilde{\varphi}_V)_{V\in\widetilde{\calU}}$, respectively.
  
  The simplicial map~$N(p)\colon N(\widetilde{\calU})\to N(\calU)$,
  under which a simplex of~$N(\widetilde{\calU})$ corresponding
  to~$V_0,\ldots,V_k\in\widetilde{\calU}$ is mapped to the simplex
  of~$N(\calU)$ corresponding to~$p(V_0),\ldots,p(V_k)\in \calU$,
  makes the following diagram commute:
  \[\begin{tikzcd}
  	\widetilde{X}\ar{r}{\widetilde{\nu}}\ar{d}{p} & {|N(\widetilde{\calU})|}\ar{d}{|N(p)|} \\
	X\ar{r}{\nu} & {|N(\calU)|} \,.
  \end{tikzcd}\]
\end{ex}

Given a pair of spaces~$(X,A)$, we introduce conditions on covers of~$X$
requiring a certain regularity near the subspace~$A$. These conditions 
will give rise to desirable properties of the lifted covers.

\begin{defn}[Relative cover]\label{defn:relative cover}
  Let $(X,A)$ be a pair of spaces with path-connected ambient space~$X$.
  A \emph{relative cover} of~$(X,A)$ is a cover~$\calU$ of~$X$ by
  path-connected open subsets such that for all~$U\in\calU$ the
  following hold:
  \begin{enumerate}[label=(RC\arabic*)]
  \item\label{item:RC1} If $U\cap A\neq \emptyset$, then $U\cap A$ is path-connected;
  \item\label{item:RC2} If $U\cap A\neq \emptyset$, then the inclusion
    \[
    \im\bigl(\pi_1(U\cap A,x)\to \pi_1(X,x)\bigr)
    \into \im\bigl(\pi_1(U,x)\to \pi_1(X,x)\bigr)
    \]
    is an isomorphism for some (whence every)~$x \in U \cap A$.
  \end{enumerate}

  A relative open cover~$\calU$ is \emph{\admissible} if for
  every~$k \in \N$ and all~$U_1, \dots, U_k \in \calU$ with
  $U_1\cap\dots\cap U_k\cap A\neq \emptyset$, each path-connected
  component of $U_1\cap\dots\cap U_k$ intersects~$A$.
\end{defn}

  We also say that a relative open cover~$\calU$ of~$(X,A)$ is \emph{convex} if the 
  underlying cover $\calU$ of $X$ is convex.
  Clearly, every convex relative cover is in particular \admissible.
	
  Given a family~$\F$ of subgroups of~$\pi_1(X)$, a relative
  cover~$\calU$ of~$(X,A)$ is a \emph{relative $\F$-cover} if the
  cover~$\calU$ of~$X$ is an $\F$-cover.

Keeping the same notation as in Example~\ref{ex:covers}, we have the
following:

\begin{prop}[Equivariant nerve pair of lifted covers]\label{prop:cover downstairs}
  Let $(X,A)$ be a CW-pair with fundamental group
  pair $(G,\calH)$.  Let $\F$ be an intersection-closed family of
  subgroups of $G$ and $\calU$ be a relative $\F$-cover of $X$. Let $\nu$ and
  $\widetilde{\nu}$ be the above nerve maps of $\calU$ and
  $\widetilde{\calU}$, respectively. 
  
  Then
  $(|N(\widetilde{\calU})|,|N_{p^{-1}(A)}(\widetilde{\calU})|)$ is a
  $(G,\calH)$-CW-pair with isotropy in $\F$
  and the map~$\widetilde{\nu}$ 
  induces a map of $(G,\calH)$-CW-pairs
  $$
  \widetilde{\nu} \colon\bigl(\widetilde{X},p^{-1}(A)\bigr)\to \bigl(|N(\widetilde{\calU})|, |N_{p^{-1}(A)}(\widetilde{\calU})|\bigr) \,
  $$
  that makes the following diagram commute:
  \[\begin{tikzcd}
  \bigl(\widetilde{X},p^{-1}(A)\bigr)\ar{r}{\widetilde{\nu}}\ar{d}{p}
  & \bigl(|N(\widetilde{\calU})|, |N_{p^{-1}(A)}(\widetilde{\calU})|\bigr)\ar{d}{|N(p)|} \\
  (X,A)\ar{r}{\nu} & \bigl(|N(\calU)|, |N_A(\calU)|\bigr) \,.
  \end{tikzcd}\]
  Moreover, we have the following:
  \begin{enumerate}[label=\enum]
  \item\label{item:dim mult} If $\calU$ is \admissible,
    then
    \[\dim\bigl(N(\widetilde{\calU}),N_{p^{-1}(A)}(\widetilde{\calU})\bigr) =
    \mult_A(\calU)-1\,;\]
    \item\label{item:convex} If $\calU$ is convex,
    then the map~$N(p)$ induces
    isomorphisms of simplicial complexes:
  \begin{align*}
  	G\backslash N(\widetilde{\calU})
	&\cong N(\calU) \,;\\
	G\backslash N_{p^{-1}(A)}(\widetilde{\calU})
	&\cong N_A(\calU) \,.
  \end{align*}
    \end{enumerate}
\end{prop}
\begin{proof}
  To show that
  $(|N(\widetilde{\calU})|,|N_{p^{-1}(A)}(\widetilde{\calU})|)$ is a
  $(G,\calH)$-CW-pair with isotropy in~$\F$, we verify that the
  lifted~$G$-invariant cover~$\widetilde{\calU}$
  of~$\widetilde X$ satisfies all assumptions of Lemma~\ref{lem:nerve
    pair}. By Example~\ref{ex:covers}, we know that
  $\widetilde{\calU}$ has isotropy in~$\F$, that there exists a
  $G$-invariant partition of unity on~$\widetilde{X}$ subordinate
  to~$\widetilde{\calU}$, and that $g\cdot V\cap V\neq\emptyset$
  implies~$g\cdot V=V$ for all~$g\in G,V\in\widetilde{\calU}$. Hence,
  we are left to show that if $V\in \widetilde{\calU}$ with $V\cap
  p^{-1}(A)\neq\emptyset$, then $V\cap p^{-1}(A)$ is connected.
	
  Assume for a contradiction that $V \cap p^{-1}(A)$ is disconnected.
  Let us set $U\coloneqq p(V)$. By condition~\ref{item:RC1} we know that
  $U\cap A$ is connected. This shows that there exists a point~$a\in
  U\cap A$ with two lifts $\widetilde{a}_1,\widetilde{a}_2$ contained
  in different components of~$V\cap p^{-1}(A)$. Since $V$ is
  path-connected, there exists a path~$\gamma$ in $V$ connecting
  $\widetilde{a}_1$ to~$\widetilde{a}_2$. By construction, the image
  of~$\gamma$ under~$p$ is a loop~$p_*\gamma$ in~$U$ based at~$a$.
  Then, by condition~\ref{item:RC2},
  the homotopy class $[p_*\gamma] \in
  \pi_1(X, a)$ admits a representative whose support is contained
  in~$U \cap A$.  Thus, there exists a lifted homotopy
  in~$\widetilde{X}$ relative to the endpoints from~$\gamma$ to a
  path in~$V\cap p^{-1}(A)$. This contradicts the fact that
  $\widetilde{a}_1$ and $\widetilde{a}_2$ lie in different components
  of~$V\cap p^{-1}(A)$.  Hence Lemma~\ref{lem:nerve pair} applies and
  yields the claim.
  
    \ref{item:dim mult} We show that $\mult_{p^{-1}(A)}(\widetilde{\calU})=
  \mult_{A}(\calU)$, which immediately implies the claim.
  The inequality $\mult_{p^{-1}(A)}(\widetilde{\calU})\ge \mult_A(\calU)$ is clear.
  To show the opposite inequality,
  let $V_1,\ldots,V_k\in\widetilde{\calU}$ with $\bigcap_{i=1}^k V_i\neq\emptyset$ 
  and $\bigcap_{i=1}^k V_i\cap p^{-1}(A)=\emptyset$.
  We claim that $\bigcap_{i=1}^k p(V_i)\cap A = \emptyset$, 
  whence $k \leq \mult_A(\calU)$ because the $(p(V_i))_{i}$ are pairwise different.  
  Indeed, assume for a contradiction that $\bigcap_{i=1}^k p(V_i)\cap
  A\neq \emptyset$. 
  Take a point~$\widetilde{x}\in \bigcap_{i=1}^k V_i$ and consider 
  $p(\widetilde{x})\in \bigcap_{i=1}^k p(V_i)$.
  Then the component of $\bigcap_{i=1}^k p(V_i)$
  containing~$p(\widetilde{x})$ intersects~$A$ by \admissibility~of~$\calU$. 
  Hence, we can
  choose a path~$\tau$ in~$\bigcap_{i=1}^k p(V_i)$ connecting $p(\widetilde{x})$ to some
  point in~$A$. Then the lifted path~$\widetilde{\tau}$ of $\tau$ in~$\widetilde{X}$ with
	starting point~$\widetilde{x}$ has endpoint in~$p^{-1}(A)$ and is 
	supported in~$\bigcap_{i=1}^k V_i$. This shows that
	 $\bigcap_{i=1}^k V_i\cap p^{-1}(A)\neq\emptyset$,
	 which is a contradiction.     
     
  	\ref{item:convex} If~$\calU$ is convex, then 
	the map~$N(p)\colon N(\widetilde{\calU})\to N(\calU)$ induces 
	the first isomorphism~$G\backslash N(\widetilde{\calU})\cong N(\calU)$
	by~\cite[Lemma~4.5~(3)]{Loeh-Sauer19}.
	Hence, to deduce the second isomorphism, it suffices to show that
	$N_{p^{-1}(A)}(\widetilde{\calU})=N(p)^{-1}(N_A(\calU))$.
	The inclusion $N_{p^{-1}(A)}(\widetilde{\calU})\subset N(p)^{-1}(N_A(\calU))$ is clear.
	To prove the opposite inclusion,
	let~$V_1,\ldots,V_k\in\widetilde{\calU}$ span a simplex in~$N(p)^{-1}(N_A(\calU))$.
	This means that $\bigcap_{i=1}^k V_i\neq\emptyset$
	and $\bigcap_{i=1}^k p(V_i)\cap A\neq \emptyset$, and
	we need to show that $\bigcap_{i=1}^k V_i\cap p^{-1}(A)\neq \emptyset$.
	Take a point~$\widetilde{x}\in \bigcap_{i=1}^k V_i$ and consider
	$p(\widetilde{x})\in \bigcap_{i=1}^k p(V_i)$. Since $\calU$ is convex
	and $\bigcap_{i=1}^k p(V_i) \cap A \neq \emptyset$,
	we can choose a path~$\tau$ in $\bigcap_{i=1}^k p(V_i)$ with 
	starting point~$p(\widetilde{x})$ and endpoint in~$A$.
	 As before, the lifted path~$\widetilde{\tau}$ of $\tau$ in $\widetilde{X}$ starting at $\widetilde{x}$ shows that
  $\bigcap_{i=1}^k V_i\cap p^{-1}(A)\neq\emptyset$.
\end{proof}

\subsection{Relative generalised LS-category}\label{subsec:relgenLS}

We introduce a relative version of the generalised
Lusternik--Schnirelmann category for families of
subgroups~\cite[Definition~2.16]{CLM20}.

\begin{defn}[Relative $\F$-category]\label{defn:factorisation}
	Let $(X,A)$ be a CW-pair with fundamental group pair~$(G,\calH)$
	and let~$p\colon \widetilde{X}\to X$ denote the universal covering. 
	Let~$\F$ be a family of subgroups of~$G$ that contains 
	the trivial subgroup.
	The \emph{relative $\F$-category} of~$(X,A)$, denoted by~$\cat_\F(X,A)$,
	is the minimal~$n\in \IN$ such that
	there exists a $(G,\calH)$-CW-pair~$(Y, B)$
	with isotropy in $\F$
	of relative dimension~$n-1$ and a 
	map of $(G,\calH)$-CW-pairs
	$(\widetilde{X},p^{-1}(A))\to (Y, B)$.
	If no such integer~$n$ exists, we set $\cat_\F(X,A)\coloneqq +\infty$.
	
	We will refer to~$\cat_\AME(X,A)$ also as the \emph{relative amenable category} of~$(X,A)$.
\end{defn}

\begin{rem}\label{rem:three defs}
	In the situation of Definition~\ref{defn:factorisation},
	let~$\EFGH$ be a model for the classifying space of the group pair~$(G,\calH)$
	with respect to the family~$\F$. Consider the (up to~$G$-homotopy unique)
	map of~$(G,\calH)$-CW-pairs
	\[
		f\colon \bigl(\widetilde{X},p^{-1}(A) \bigr)\to \EFGH \,.
	\]
	Let~$n\in \IN$. Then the following are equivalent:
	\begin{enumerate}[label=\enum]
		\item We have $\cat_\F(X,A)\le n$;
		\item The map~$f$ factors (up to~$G$-homotpy) through a~$(G,\calH)$-CW-pair~$(Y,B)$ 
		with isotropy in~$\F$ of relative dimension~$n-1$;
		\item The map~$f$ is~$G$-homotopic to a map of~$(G,\calH)$-CW-pairs with values
		in the relative $(n-1)$-skeleton of~$\EFGH$.
	\end{enumerate}
	Indeed, the equivalence of these conditions follows from
	the universal property of~$\EFGH$ and the equivariant cellular approximation
	theorem~\cite[Theorem~2.1]{lueck_TG}.
\end{rem}

By definition, the relative $\F$-category satisfies $\cat_\F(X,A)\le \dim(X,A)+1$.
A more efficient upper bound for the relative category is provided by 
the existence of \admissible\ relative covers:

\begin{defn}[Relative $\F$-\category]\label{defn:relative category}
  Let $(X, A)$ be a pair of spaces and let $\F$ be a family of
  subgroups of $\pi_1(X)$.  The \emph{relative $\F$-\category} of $(X,
  A)$, denoted by~$\relcat_{\F}(X, A)$, is the minimal~$n \in \N$
  such that there exists a \admissible\ relative $\F$-cover~$\calU$
  of~$(X, A)$ with~$\mult_A(\calU) = n$.  If no such integer~$n$ exists, we
  set $\relcat_{\F}(X, A) \coloneqq +\infty$.
	
  We will refer to~$\relcat_{\AME}(X, A)$ also as the
  \emph{relative amenable \category} of~$(X,A)$.
\end{defn}

\begin{lem}\label{lem:factorisation multiplicity}
	Let $(X,A)$ be a CW-pair with fundamental group 
	pair~$(G,\calH)$.
	Let~$\F$ be a family of subgroups of~$G$ that contains 
	the trivial subgroup. Then we have
	\[
		\cat_\F(X,A)\le \relcat_\F(X,A)\,.
	\]
	\begin{proof}
		We may assume that $n\coloneqq \relcat_\F(X,A)$ is finite.
		Let~$\calU$ be a \admissible\ relative~$\F$-cover of~$(X,A)$ with
		$\mult_A(\calU)=n$. By Proposition~\ref{prop:cover downstairs}, 
		the equivariant nerve pair~$(|N(\widetilde{\calU})|,|N_{p^{-1}(A)}(\widetilde{\calU})|)$ of the lifted cover~$\widetilde{\calU}$ of~$\widetilde{X}$
		is a~$(G,\calH)$-CW-pair with isotropy in $\F$ and of relative dimension~$n-1$.
        Hence the nerve map $$
  \widetilde{\nu} \colon\bigl(\widetilde{X},p^{-1}(A)\bigr)\to \bigl(|N(\widetilde{\calU})|, |N_{p^{-1}(A)}(\widetilde{\calU})|\bigr) \,
  $$
  exhibits the desired inequality.
	\end{proof}
\end{lem}

\begin{rem}[Category and multiplicity, absolute case]\label{rem:abscat}
	Let~$X$ be a path-connected CW-complex with fundamental group~$G$
	and let~$\F$ be a family of subgroups of~$G$.
	In the absolute case, the generalised Lusternik--Schnirelmann 
	category~$\cat_\F(X)$ is defined as the minimal~$n$ for which 
	there exists an open~$\F$-cover of~$X$ by $n$ many
	\emph{not necessarily path-connected} subsets.
	If the family~$\F$ is closed under taking subgroups,
	this is compatible with Definition~\ref{defn:factorisation} in the sense that
	$\cat_\F(X,\emptyset)=\cat_\F(X)$ by~\cite[Lemma~7.6]{CLM20}.
	
	In particular, also the converse estimate of 
	Lemma~\ref{lem:factorisation multiplicity} holds in the absolute case:
	Indeed, taking path-connected components of open~$\F$-covers with~$n$ 
	not necessarily path-connected members produces an~$\F$-cover of multiplicity
	at most~$n$. Therefore, we obtain~$\cat_\F(X,\emptyset)=\mult_\F(X,\emptyset)$.
\end{rem}

If $f \colon (Z,C) \to (X,A)$ is a homotopy equivalence of CW-pairs,
then pulling back fundamental group pairs and families of subgroups
along~$\pi_1(f)$ shows that $\cat_{\pi_1(f)^*\F}(Z,C) =
\cat_\F(X,A)$. In contrast, it is not clear whether $\mult_{\F}$ is
also a relative homotopy invariant.

\section{Simplicial volume, bounded cohomology, and acyclicity}\label{sec:prel:sv:bc:bac}

In this section we recall the notions of simplicial volume and bounded
cohomology. We also recall bounded acyclicity and we introduce a
uniform version of bounded acyclicity. In particular, we explain how
uniformly boundedly acyclic actions lead to computations of bounded
cohomology.  This is an adaptation of standard techniques in bounded
cohomology~\cite{Monod,MR:BAc}; similar results are also discussed in
recent computations of bounded cohomology groups~\cite{fffclmm2,
  monodnariman}.

\subsection{Simplicial volume}

We recall the definition of simplicial volume~\cite{vbc}.  Let $(X,
A)$ be a pair of spaces. For every singular $n$-chain $c = \sum_{i =
  1}^k a_i \sigma_i \in C_n(X, A; \R)$, written in reduced form, we
define the $\ell^1$-\emph{norm} as follows:
$$
| c |_1 \coloneqq \sum_{i = 1}^k |a_i| \,.
$$
The restriction of the $\ell^1$-norm to the subspace of relative
cycles induces a quotient $\ell^1$-seminorm (denoted
by~$\|\cdot\|_1$) on the homology group~$H_n(X, A; \R)$.

\begin{defn}[Relative simplicial volume]
  Let $M$ be an oriented connected compact $n$-manifold with (possibly
  non-empty) boundary.  Then the \emph{relative simplicial volume} of
  $M$ is
  $$
  \sv{M, \partial M} \coloneqq \bigl\| [M, \partial M] \bigr\|_1 \,,
  $$
  where $[M, \partial M] \in H_n(M, \partial M; \R) \cong \R$ denotes
  the relative fundamental class of $M$.
\end{defn}

\begin{ex}\label{ex:relative:sv}
  Let $M$ be an oriented compact connected $n$-manifold.
  \begin{enumerate}[label=\enum]
  \item If $M^\circ$ admits a complete finite-volume hyperbolic
    metric, then we have $\sv{M, \partial M} = \textup{vol}(M) \slash
    v_n$~\cite{vbc, FujiwaraManning};

  \item If $M$ is a handlebody of genus $g \geq 2$, then $\sv{M,
    \partial M} = 3 \cdot (g -1)$~\cite{BFP};

  \item If $M = \Sigma_g \times I$, where $\Sigma_g$ is a surface of
    genus $g \geq 2$, then we have $\sv{M, \partial M} = \frac{5}{4} \cdot
    \sv{\partial M}$~\cite{BFP};
  \item 
    The simplicial volume of graph manifolds is
    zero~\cite{Soma, vbc}.
  \item
    If $M$ admits a self-map~$f$ with~$\mathopen|\deg(f)\mathclose| \geq 2$,
    then~$\sv{M,\partial M} = 0$~\cite{vbc}.
  \item
    If $M$ is closed and admits an open cover by amenable subsets of
    multiplicity at most~$\dim (M)$, then $\| M \| = 0$~\cite{vbc}.
    Many examples are known to satisfy this
    condition~\cite[Section~1.1]{loehmoraschinisauer}.
  \end{enumerate}
  Further computations of simplicial volume are surveyed in the
  literature~\cite{LMR}.
\end{ex}

We also recall the locally finite version of simplicial volume
for non-compact manifolds~\cite{vbc, Loeh, FM:Grom}.  Given a
topological space $X$, a (possibly infinite) real singular
$n$-chain~$c = \sum_{\sigma \in \map(\Delta^n, X)} a_\sigma \sigma$ is
\emph{locally finite} if every compact subset of~$X$ intersects only
finitely many simplices with non-trivial coefficient. We define
$C^{\lf}_*(X; \R)$ as the $\R$-module of locally finite chains
on~$X$. The usual boundary operator for finite chains admits a canonical extension
to~$C^{\lf}_*(X;\R)$.  The \emph{locally finite
  homology}~$H_*^{\lf}(X; \R)$ of $X$ is the homology of the complex
$C_*^{\lf}(X; \R)$.

As in the finite case, the $\ell^1$-norm of a locally finite chain $c
= {\sum_{\sigma \in \map(\Delta^n, X)} a_\sigma \sigma}$
in~$C_n^{\lf}(X; \R)$ is given by
$$
| c |_1 \coloneqq \sum_{\sigma \in \map(\Delta^n, X)} |a_\sigma| \in [0, +\infty] \,.
$$
As before, this norm induces an $\ell^1$-seminorm $\| \cdot \|_1$
on~$H_n^{\lf}(X; \R)$.

\begin{defn}[Locally finite simplicial volume]
  Let $M$ be an oriented (possibly non-compact) connected $n$-manifold
  without boundary.  The \emph{locally finite simplicial volume}
  of~$M$ is defined by
  $$
  \sv{M}_{\lf} \coloneqq \bigl\| [M]_{\lf} \bigr\|_1 \,,
  $$
  where $[M]_{\lf} \in H_n^{\lf}(M; \R) \cong \R$ denotes the locally
  finite fundamental class.
\end{defn}

  The [locally finite] simplicial volume can be defined for
  every normed ring $R$. In this case, we will consider the
  $\ell^1$-seminorm on $H^{[\lf]}_*(-; R)$ and we will talk about
  [\emph{locally finite}] $R$-\emph{simplicial volume} $\sv{\cdot}_{R,
    [\lf]}$.

\subsection{Bounded cohomology}

We recall the definition of bounded cohomology of groups and
spaces~\cite{vbc,ivanov,Monod,Frigerio} as well as its equivariant
version~\cite{Loeh-Sauer19,Li21}.

For a group~$G$ and a normed $\R G$-module~$V$, we write
\[ C_b^*(G;V) \coloneqq \linf(G^{*+1},V)^G
\]
(equipped with the simplicial coboundary operator) for the
\emph{bounded cochain complex of~$G$ with coefficients in~$V$}.

\begin{defn}[Bounded cohomology of groups]
  The \emph{bounded cohomology of~$G$ with coefficients in~$V$} is
  defined by
  \[ H_b^*(G;V) \coloneqq H^*\bigl( C_b^*(G;V) \bigr) \,.
  \]
\end{defn} 

Similarly, if $(X, A)$ is a topological pair, we can consider the
singular cochain complex
$$
C^*(X, A; \R) \coloneqq\bigl\{f \in C^*(X; \R)
\bigm| \text{$f(\sigma) = 0$ for all $\sigma$ supported in $A$}
\bigr\} \,,
$$
where a singular $n$-simplex $\sigma$ is \emph{supported in~$A$}
if $\sigma(\Delta^n) \subset A$ \,.
We can restrict to the subcomplex of bounded cochains:
$$
C_b^*(X, A; \R) \coloneqq
\bigl\{f \in C^*(X, A; \R)
\bigm| \text{$\textstyle\sup_{\sigma \in \map(\Delta^*,X)} |f(\sigma)| < \infty$}
\bigr\} \,.
$$

\begin{defn}[Bounded cohomology of spaces]
  Let $(X,A)$ be a pair of spaces.  The \emph{bounded cohomology
    of~$(X, A)$} (with real coefficients) is defined by
  $$
  H_b^*(X, A;\R) \coloneqq H^*\bigl( C_b^*(X, A;\R) \bigr) \,.
  $$
\end{defn}

The inclusion of complexes $C_b^*(X, A; \R) \hookrightarrow C^*(X, A;
\R)$ induces a natural map from bounded cohomology to ordinary
cohomology, the \emph{comparison map}:
$$
\comp_{(X, A)}^* \colon H_b^*(X, A; \R) \to H^*(X, A; \R) \,.
$$

The connection between bounded cohomology and simplicial volume is
encoded in the following classical result:

\begin{prop}[{Duality principle, qualitative version~\cite{vbc, Frigerio}}]\label{duality:principle}
  Let $M$ be an oriented connected compact $n$-manifold with (possibly
  empty) boundary.  Then the following are equivalent:
  \begin{enumerate}[label=\enum]
  	\item $\sv{M, \partial M} > 0$;
	\item The comparison map $\comp_{(M, \partial M)}^n$ is surjective.
  \end{enumerate}
\end{prop}

We also recall the equivariant version of
bounded cohomology~\cite[Definition~5.1]{Loeh-Sauer19}:

\begin{defn}[{Equivariant [bounded] cohomology}]\label{defn:equivariant bounded cohomology}
  Let $Y$ be a~$G$-space and let~$C_*(Y; \R)$ denote the singular chain
  complex.  For coefficients in a [normed]~$\IR G$-module~$V$, we
  define the cochain complex
  \[
  C^*_G(Y;V)\coloneqq \Hom_{\IR G}\bigl(C_*(Y;\IR),V\bigr)
  \]
  and the subcomplex~$C^*_{G,b}(Y;V)\subset C^*_G(Y;V)$ consisting of
  [bounded]~$\IR G$-homo\-mor\-phisms. Then we set
  \begin{align*}
    H^n_G(Y;V) &\coloneqq H^n\bigl(C^*_G(Y;V)\bigr) \,; \\
    H^n_{G,b}(Y;V) &\coloneqq H^n\bigl(C^*_{G,b}(Y;V)\bigr) \,.
  \end{align*}
  For a pair of~$G$-spaces~$(Y,B)$ one similarly
  defines~$H^n_G(Y,B;V)$ and~$H^n_{G,b}(Y,B;V)$.
  As in the absolute case, there is a \emph{comparison map}
  \[
  	\comp^n_{G,(Y,B)}\colon H^n_{G,b}(Y,B;V)\to H^n_G(Y,B;V) \,.
  \]
\end{defn}

We have the following induction isomorphisms:

\begin{lem}\label{lem:induction isomorphism}
  Let $H$ be a subgroup of~$G$, let $B$ be an~$H$-space, and let $n
  \in \N$.  Then there are natural isomorphisms of $\R$-vector spaces:
  \begin{align*}
    H^n_G(G\times_H B;\IR) &\xrightarrow{\cong} H^n_H(B;\IR) \,; \\
    H^n_{G,b}(G\times_H B;\IR) &\xrightarrow{\cong} H^n_{H,b}(B;\IR) \,. 
  \end{align*}
\end{lem}
\begin{proof}
  Let $C_*(-)$ denote the singular chain complexes (with real
  coefficients).  Since the induced $G$-space~$G\times_H B$ consists of disjoint copies
  of~$B$, the image of a singular simplex in~$G\times_H B$ is
  contained in a single copy.  Hence we have a natural
  isomorphism~$C_*(G\times_H B)\cong \IR G\otimes_{\IR H}C_*(B)$ of
  $\IR G$-chain complexes and thus an adjunction isomorphism
  \[
  \Phi\colon C^*_G(G\times_H B;\IR)\xrightarrow{\cong} C^*_H(B;\IR) \,.
  \]
  This yields the isomorphism on equivariant cohomology.  The
  claim on equivariant bounded cohomology follows by observing
  that $\Phi$ restricts to an isomorphism
  $
  C^*_{G,b}(G\times_H B;\IR)\to C^*_{H,b}(B;\IR) 
  $ 
  on the subcomplexes of bounded cochains.
\end{proof}

\subsection{Bounded acyclicity}

We recall the definition of bounded acyclicity for modules and
groups:

\begin{defn}[Boundedly acyclic group]\label{defn:BAc module}
  Let $G$ be a group and let $n \in \N$.  A normed~$\IR G$-module~$V$
  is \emph{boundedly~$n$-acyclic} if~$H^k_b(G;V)\cong 0$ for all~$k\in
  \{1,\ldots,n\}$.  A normed~$\IR G$-module~$V$ is \emph{boundedly
    acyclic} if~$H^k_b(G;V)\cong 0$ for all~$k\in \IN_{\geq 1}$.
	
  The group~$G$ is \emph{boundedly~$n$-acyclic}
  [resp.~\emph{boundedly acyclic}] if the trivial~$\IR
  G$-module~$\IR$ is boundedly~$n$-acyclic [resp.~boundedly acyclic].
\end{defn}

Amenable groups are boundedly acyclic; by now, there is a wide range
of known examples of non-amenable boundedly acyclic groups, including
finitely presented
examples~\cite{matsumotomorita,loeh_bcd,fffclmm,fffclmm2,monodnariman,monodthompson}.

Resolutions by boundedly acyclic modules can be used to compute 
bounded cohomology~\cite[Proposition~2.5.4 and Remark~2.5.5]{MR:BAc}:

\begin{thm}[Fundamental lemma for boundedly acyclic resolutions~\cite{MR:BAc}]
  \label{thm:fundamental lemma}
  Let $G$ be a group and let $n \in \N$.  Let $0\to V\to V^*$ be a
  resolution of normed~$\IR G$-modules
  such that $V^j$ is a dual normed $\R G$-module and
  boundedly~$(n-j)$-acyclic for every~$j \in \{0,\dots,n-1\}$.  Then
  there is a canonical isomorphism (of $\R$-vector spaces)
  \[
  H^k(V^*{}^G)\xrightarrow{\cong} H^k_b(G;V)
  \]
  for all~$k \in \{0,\dots,n\}$ and a canonical injective map
  \[
  H^{n+1}(V^*{}^G)\into H^{n+1}_b(G;V) \,.
  \]
  Moreover, if the given resolution is strong, then these
  maps are the ones induced by the canonical $G$-cochain
  homotopy class~$V^* \to \linf(G^{*+1},V)$.
\end{thm}

\subsection{Uniform bounded acyclicity}\label{subsec:UBAc collections}

To formulate uniform bound\-ed acyclicity for collections of groups, 
we need additional control
on the norms of primitives. This can be expressed in terms of the
uniform boundary condition (see Appendix~\ref{appx:ubc} for a
definition and some properties).
More precisely, we use the \emph{uniform uniform boundary condition}
denoted by~$\UUBC$ (Definition~\ref{defn:UUBC}).

\begin{defn}[Uniformly boundedly acyclic collection of groups]\label{def:ubacgroup}
  A collection~$\calG$ of groups is \emph{uniformly boundedly acyclic}
  if
  \begin{itemize}
  \item all members of~$\calG$ are boundedly acyclic and
  \item the collection~$(C_b^*(H;\R))_{H \in \calG}$
    satisfies~$\UUBC^k$ for all~$k \in \N$.
  \end{itemize}
  Similarly, for~$n \in \N$, we define \emph{uniformly 
  boundedly~$n$-acyclic} collections of groups if the previous conditions are
  satisfied up to degree~$n$.
  Moreover, we extend these definitions to sets of groups.
\end{defn}

For example, all collections consisting of amenable groups 
are uniformly boundedly acyclic (Example~\ref{ex:amenable UUBC}). 
Also, all \emph{finite} collections of boundedly acyclic groups
are uniformly boundedly acyclic (Example~\ref{exa:uubcgroup}).

\begin{prop}\label{prop:ubacshapiro}
  Let $n \in \N$, let $G$ be a group, let $(H_i)_{i\in I}$ be a
  uniformly boundedly~$n$-acyclic collection of subgroups of~$G$, and
  let $k \in \{1,\dots,n\}$. Then
  \[ H_b^k\bigl(G; \linf(G/H_i,\R) \bigr) \cong 0
  \]
  for all~$i \in I$ and the collection~$(C_b^*(G;\linf(G/H_i,\R)))_{i \in
    I}$ satisfies~$\UUBC^k$.
\end{prop}

\begin{proof}
  This is a boundedly controlled version of the Shapiro lemma: By the
  Shapiro lemma in bounded
  cohomology~\cite[Proposition~10.1.3]{Monod}, we have
  \[ H_b^k\bigl(G; \linf(G/H_i,\R) \bigr) \cong H_b^k(H_i;\R) \cong 0
  \]
  for all~$i \in I$ and all $k \in \{1, \dots, n\}$. In order to
  conclude that $(C_b^*(G;\linf(G/H_i,\R)))_{i \in I}$
  satisfies~$\UUBC^k$, we make the proof of the Shapiro lemma more
  explicit:

  Let $H \subset G$ be a subgroup of~$G$. Then there is a non-empty
  set~$J$ such that $G$, as an $H$-space, is isomorphic to~$J \times
  H$ (with the translation action on the $H$-factor). Therefore, on
  the one hand, we obtain isometric isomorphisms
  \begin{align*}
    C_b^*\bigl( G; \linf(G/H,\R)\bigr)
    & \cong \linf\bigl(G^{*+1}, \linf(G/H,\R)\bigr)^G
    \cong \linf\bigl(\res^G_H G^{*+1}, \R\bigr)^H
    \\
    & \cong \linf\bigl( (J \times H)^{*+1}, \R\bigr)^H
  \end{align*}
  of cochain complexes (each equipped with the simplicial coboundary
  operator). On the other hand, $C_b^*(H ;\R) = \linf(
  H^{*+1},\R)^H$.  Both sides are connected through mutually
  homotopy inverse cochain homotopy equivalences
  \[ \linf\bigl( H^{*+1},\R\bigr)^H \leftrightarrow \linf\bigl( (J \times H)^{*+1}, \R\bigr)^H
  \]
  given by (where $0 \in J$ is a chosen basepoint)
  \begin{align*}
    \varphi^* \colon 
    \linf\bigl( H^{*+1},\R\bigr)^H
    & \to \linf\bigl( (J \times H)^{*+1}, \R\bigr)^H
    \\
    f
    & \mapsto \bigl( ((i_0, h_0), \dots, (i_k,h_k)) \mapsto
    f(h_0, \dots, h_k) \bigr)
    \\
    \psi^* \colon 
    \linf\bigl( (J \times H)^{*+1}, \R\bigr)^H
    & \to \linf\bigl( H^{*+1},\R\bigr)^H
    \\
    f
    & \mapsto
    \bigl(
    (h_0, \dots, h_k) \mapsto
    f\bigl((0,h_0), \dots, (0,h_k) \bigr)
    \bigr)
    \,;
  \end{align*}
  these cochain maps have norm~$1$ in each degree.  Indeed, $\psi^*
  \circ \varphi^*$ is the identity on~$\linf(H^{*+1}, \R)^H$ and the
  standard map
  \begin{align*}
    \linf\bigl( (J \times H)^{*+1}, \R\bigr)^H
    & \to
    \linf\bigl( (J \times H)^{*}, \R\bigr)^H
    \\
    f & \mapsto
    \biggl( ((i_0, h_0), \dots, (i_{k-1}, h_{k-1}))\mapsto
    \\ & \quad\quad
    \sum_{j=0}^{k-1} (-1)^j \cdot f\bigl((i_0, h_0), \dots,
    (i_j,h_j), (0,h_j), \dots, (0,h_{k-1}) \bigr) \biggr)
  \end{align*}
  is a cochain homotopy between $\varphi^* \circ \psi^* $ and the
  identity on~$\linf((J\times H)^{*+1},\R)^H$, with norm~$k$ in
  degree~$k$.  In particular, all these norms are independent of the
  subgroup~$H \subset G$.

  Hence, the claim follows by applying these considerations and
  homotopy inheritance of~$\UBC$
  (Proposition~\ref{prop:ubchinherit}) to the subgroups~$(H_i)_{i \in
    I}$ of~$G$.
\end{proof}

\subsection{Uniformly boundedly acyclic actions}\label{subsec:UBAc modules}

Group actions with amenable stabilisers, so-called amenable actions,
have proved to be a valuable tool to compute bounded cohomology in
specific cases~\cite{Monod, burgermonodGAFA, BIuseful}. Similarly,
also uniformly boundedly acyclic actions allow us to compute
bounded cohomology.  This is an easy application of the fact that
bounded cohomology can be computed via acyclic
resolutions (Theorem~\ref{thm:fundamental lemma}).  However, usually,
in this approach we cannot compute the seminorm on bounded cohomology.

\begin{defn}[Uniformly boundedly acyclic action]
  A group action on a set is \emph{uniformly boundedly acyclic} if the
  collection of all stabilisers forms a uniformly boundedly acyclic
  collection of groups.  Similarly, for~$n \in \N$, we introduce the
  notion of \emph{uniformly boundedly~$n$-acyclic actions}.
\end{defn}

Uniformly boundedly acyclic actions lead to boundedly acyclic modules:

\begin{prop}\label{prop:stabacyc}
  Let $G$ be a group, let $n \in \N$, and let $G \actson S$ be a
  uniformly boundedly~$n$-acyclic action on a set~$S$. Then, for
  all~$k \in \{1,\dots, n\}$, we have
  \[ H_b^k\bigl( G ; \linf(S,\R) \bigr)
     \cong 0 \,.
  \]
\end{prop}
\begin{proof}
  Without loss of generality, we may assume that $S = \coprod_{i \in
    I} G/H_i$ with the left translation action on each summand.  Using
  the uniform version of the Shapiro lemma
  (Proposition~\ref{prop:ubacshapiro}) and the compatibility with
  bounded products (Theorem~\ref{thm:bprod}), we obtain for every $k
  \in \{1, \dots, n\}$
  \begin{align*}
    0
    & \cong \bprod{i \in I} H_b^k\bigl(G; \linf(G/H_i,\R)\bigr)
    & \text{(Proposition~\ref{prop:ubacshapiro})}
    \\
    & \cong H^k \Bigl( \bprod{i \in I} C_b^*\bigl(G; \linf(G/H_i,\R)\bigr) \Bigr)
    & \text{(Theorem~\ref{thm:bprod})}
    \\
    & \cong H^k \Bigl( C_b^*\Bigl(G; \bprod{i \in I} \linf(G/H_i,\R)\Bigr)\Bigr)
    & \text{(direct computation)}
    \\
    & \cong H^k \Bigl( C_b^*\Bigl(G; \linf\Bigl(\coprod_{i \in I} G/H_i,\R\Bigr) \Bigr)
    & \text{(Example~\ref{exa:bprodcoprod})}
    \\
    & = H_b^k\bigl(G; \linf(S,\R)\bigr) \,,
  \end{align*}
  as claimed.
\end{proof}

\begin{cor}[Bounded cohomology via uniformly boundedly acyclic actions]\label{cor:bac:set:computes:bounded:cohomology}
  Let $G$ be a group, let $G \actson S$ be an action on a non-empty
  set~$S$.  Let $n \in \N_{> 0}$ and suppose that the diagonal action
  $G \actson S^n$ is uniformly boundedly $n$-acyclic.  Then the
  cohomology of the simplicial cochain complex~$\linf(S^{*+1},\R)^G$ is
  canonically isomorphic to $H^*_b(G; \R)$ in all degrees~$\leq n$ and
  there exists a canonical injective map
  $$
  H^{n+1}\bigl(\linf(S^{*+1},\R)^G\bigr) \hookrightarrow H^{n+1}(G; \R) \,.
  $$
  More precisely, every $G$-cochain map~$\linf(S^{*+1},\R)
  \to \linf(G^{*+1},\R)$ that is degree-wise bounded and extends~$\id_\R \colon \R \to \R$ induces an
  isomorphism \textup{[}resp.~injection\textup{]} $H^k(\linf(S^{*+1},\R)^G) \to H_b^k(G;\R)$ in the
  corresponding range for~$k$.
\end{cor}
\begin{proof}
  Bounded cohomology can be computed through boundedly acyclic
  resolutions (Theorem~\ref{thm:fundamental lemma}).  As $S$ is
  non-empty, $\linf(S^{*+1},\R)$ is a strong resolution of~$\R$ by normed
  $\R G$-modules.  Therefore, it suffices to notice that in the
  present situation the $\R G$-modules~$\linf(S^{k+1}, \R)$ are boundedly
  acyclic (Definition~\ref{defn:BAc module}) for every~$k\in \{0,\ldots,n-1\}$ by
  Proposition~\ref{prop:stabacyc}.
\end{proof}

\begin{rem}[Amenable actions]
  Proposition~\ref{prop:stabacyc} and
  Corollary~\ref{cor:bac:set:computes:bounded:cohomology} are
  analogous to the corresponding results for amenable
  actions~\cite{Monod}\cite[Section~4.9]{Frigerio}: 
  If the action of~$G$ on a set~$S$ is amenable, then the
  normed~$\IR G$-module~$\linf(S,\R)$ is relatively injective 
  and hence the cochain complex~$\linf(S^{*+1},\R)^G$
  computes~$H^*_b(G;\R)$.
\end{rem}

\begin{rem}[Bounded cohomology via alternating cochains]\label{rem:alt:computes:bc}
  Let $G$ be a group and let $G \actson S$ be an action on a non-empty
  set~$S$.  A bounded function $f \in \linf(S^k, \R)$ is
  \emph{alternating} if
  $$
  f\bigl(s_{\sigma(1)}, \dots, s_{\sigma(k)}\bigr) = \mbox{sign}(\sigma) \cdot f(s_1, \dots, s_k) 
  $$
  holds for every permutation~$\sigma \in \Sigma_k$ and all~$(s_1,
  \dots, s_k) \in S^k$.  We write
  $$\linf_{\alt}(S^{*+1}, \R) \subset \linf(S^{*+1}, \R)$$
  for the subcomplex of alternating functions, which is well-defined since the coboundary operator
  preserves being alternating.

  Let $n \in \N_{> 0}$ and suppose that the diagonal action $G \actson
  S^n$ is uniformly boundedly $n$-acyclic. Then also the cohomology
  of the simplicial cochain complex
  $$
  \linf_{\alt}(S,\R)^G \rightarrow \linf_{\alt}(S^2,\R)^G \rightarrow \linf_{\alt}(S^3, \R)^G \rightarrow \cdots
  $$
  is canonically isomorphic to $H^*_b(G; \R)$ in all degrees~$\leq n$ and 
  the canonical map 
  $
  H^{n+1}(\linf_{\alt}(S^{*+1},\R)^G) \to H^{n+1}(G; \R)
  $
  is injective.
  Indeed, by Corollary~\ref{cor:bac:set:computes:bounded:cohomology},
  we already know that the previous result holds for the
  non-alternating complex.  Moreover, the
  inclusion~$\linf_{\alt}(S^{*+1},\R) \hookrightarrow \linf(S^{*+1},\R)$
  induces an isomorphism on cohomology; this can be seen from the same
  computation as in the case of the
  complex~$\linf(G^{*+1},\R)$~\cite[Proposition~4.26]{Frigerio}.
\end{rem}

We conclude this section by showing that the computation of bounded
cohomology via alternating cochains of boundedly acyclic actions is
natural in the following sense.  This is analogous to the case of
amenable actions~\cite[Lemma~2.2]{BBFIPP}.

\begin{lem}\label{lemma:prel:induced:at:level:of:group:is:the:same}
  Let $i \colon H \to G$ be a group homomorphism. Let $H \actson S_H$
  and $G \actson S_G$ be actions on non-empty sets $S_H$ and~$S_G$,
  respectively. Let $\varphi \colon S_H \to S_G$ be an $i$-equivariant
  map.
  \begin{enumerate}[label=\enum]
  \item\label{item:UBAc action alt naturality}
    Then the following diagram commutes, where the horizontal arrows
    are the canonical maps (induced by restriction to a single orbit):
    $$
    \xymatrix{
      H^*\bigl(\linf_{\alt}(S_G^{*+1}, \R)^G\bigr) \ar[d]^{H^*(\varphi^{*+1})} \ar[r]
      & H^*_b(G; \R) \ar[d]^{H^{*}_b(i)}
      \\
      H^*\bigl(\linf_{\alt}(S_H^{*+1}, \R)^H\bigr) \ar[r]
      & H^*_b(H; \R) \,.
    }
    $$
  \item\label{item:UBAc action alt}
    Let $n \in \N_{> 0}$ and suppose that the diagonal actions $H
    \actson S_H^n$ and $G \actson S_G^n$ are uniformly boundedly
    $n$-acyclic.  Then the horizontal arrows are isomorphisms in all
    degrees~$\leq n$ and injective in degree~$n+1$.
\end{enumerate}
\end{lem}
\begin{proof}
  Part~\ref{item:UBAc action alt} is shown in Remark~\ref{rem:alt:computes:bc}.

  Part~\ref{item:UBAc action alt naturality} is a straightforward computation: As $S_H$ and $S_G$
  are non-empty, we can choose a point~$x_H \in S_H$ and set $x_G \coloneqq
  \varphi(x_H)$.  The orbit maps~$\psi_H \colon H \to S_H$ for~$x_H$
  and $\psi_G \colon G \to S_G$ for~$x_G$ induce cochain maps
  extending~$\id_\R \colon \R \to \R$ and therefore induce the
  canonical maps~$H^*(\linf_{\alt}(S_H^{*+1},\R)^H) \to H_b^*(H;\R)$ and
  $H^*(\linf_{\alt}(S_G^{*+1},\R)^G) \to H_b^*(G;\R)$, respectively.
  Because $\varphi$ is $i$-equivariant and because $\varphi(x_H) =
  x_G$, the diagram
  \[ \xymatrix{%
    \linf_{\alt}(S^{*+1}_G,\IR)^G\ar[r]^-{\psi_G^{*+1}} \ar[d]_-{\varphi^{*+1}}
    & C^*_b(G;\IR)\ar[d]^-{C_b^*(i;\R)} \\
    \linf_{\alt}(S^{*+1}_H,\IR)^H\ar[r]_-{\psi_H^{*+1}} & C^*_b(H;\IR) \,,
  }
  \]
  commutes. Taking cohomology proves part~\ref{item:UBAc action alt naturality}.
\end{proof}

\section{A vanishing theorem for relative simplicial volume}

We prove vanishing theorems for the comparison map and for relative
simplicial volume in the presence of uniformly boundedly acyclic open
covers with small multiplicity, as outlined in
Section~\ref{subsec:intro:vanishing}. The proof uses equivariant nerve
pairs and equivariant bounded cohomology with respect to families of
boundedly acyclic subgroups.

\subsection{Uniformly boundedly acyclic open covers}\label{subsec:UBAc covers}

We introduce the notion of uniformly boundedly acyclic open covers.

\begin{defn}[Families associated to a set of subspaces]
  Let~$X$ be a path-connected space with~$\pi_1(X)=G$ and~$\calA$ be a
  set of path-connected subspaces of~$X$.  Consider the set
  \[
  \calG\coloneqq \bigl\{\im\bigl(\pi_1(A\into X)\bigr)\bigm| A\in\calA\bigr\}
  \]
  of subgroups of~$G$; again, we implicitly use the convention on
  basepoints (Section~\ref{subsec:convs}).  Then the
  \emph{intersection-closed family of subgroups of $G$ associated to
    $\calA$} is defined as
  \[
  \F\spann{\calA}\coloneqq \F\spann{\calG} = \biggl\{\bigcap_{i=1}^n g_iH_ig_i^{-1}
  \biggm| n \in \N, H_i\in \calG,g_i\in G \biggr\} \,.
  \]
  For fixed $n\in \IN$, we define the (conjugation-closed) family of
  subgroups of $G$ associated to $\calA$ as
  \[
  \F_n\spann{\calA}\coloneqq \F_n\spann{\calG} = \biggl\{\bigcap_{i=1}^n g_iH_ig_i^{-1}
  \biggm| H_i\in \calG,g_i\in G\biggr\} \,.
  \]
\end{defn}

Using the notion of uniformly boundedly acyclic collections of groups
(Definition~\ref{def:ubacgroup}), we define the following:
  
\begin{defn}[Uniformly boundedly acyclic set of subspaces]
  \label{defn:UBAc set of subspaces}
  Let~$X$ be a path-connected space and~$\calA$ be a set of
  path-connected subspaces of~$X$.  We say that~$\calA$ is
  \emph{uniformly boundedly acylic} [\emph{of order~$n$}] \emph{in~$X$} if the associated
  family~$\F\spann{\calA}$ [resp.~$\F_n\spann{\calA}$]
  is uniformly boundedly acyclic.
\end{defn}

\begin{defn}[Uniformly boundedly acyclic open cover]\label{defn:UBAc cover}
  Let~$X$ be a path-connected space and~$\calU$ be a cover of~$X$ by
  path-connected open subsets.  We say that~$\calU$ is \emph{uniformly
    boundedly acyclic} if it is uniformly boundedly acyclic in~$X$ when
  viewed as a set of subspaces of~$X$.
\end{defn}

By Example~\ref{ex:amenable UUBC}, every amenable cover is uniformly boundedly acyclic.

\begin{rem}\label{rem:Ivanov}
  The above notion of uniformly boundedly acyclic open covers is
  similar to Ivanov's notion of weakly boundedly acyclic open
  covers~\cite[Section~4]{ivanov_bac_covers}.  The difference is that
  we consider intersections of subgroups of the fundamental group
  whereas Ivanov considers fundamental groups of intersections of
  subspaces.  The key steps of our arguments happen on the level of
  bounded cohomology. In contrast, Ivanov's arguments via spectral
  sequences target the comparison map more directly.  It is not clear
  to us whether one of the concepts contains the other.
\end{rem}

\subsection{Strong $H_b^*$-admissibility}\label{subsec:strongadm}

We show that classifying spaces with respect to uniformly boundedly
acyclic families can be used to compute bounded cohomology.
This is phrased in terms of equivariant bounded cohomology
(Definition~\ref{defn:equivariant bounded cohomology}).

\begin{defn}[Strongly $H_b^*$-admissible family]\label{defn:admissible family}
  Let $G$ be a group and let $\F$ be a (conjugation-closed) family of
  subgroups of~$G$ that contains the trivial subgroup~$1$.  Since~$1
  \in \F$, there is a canonical (up to $G$-homotopy) $G$-map~$f \colon
  E G \to E_\F G$.
  The family~$\F$ is \emph{strongly $H_b^*$-admissible} if the induced
  map
  \[ H^*_{G,b}(f;\R)
  \colon H^*_{G,b}(E_\F G;\R)
  \to H^*_{G,b}(E G;\R)
  \cong H_b^*(G;\R)
  \]
  is bijective.
\end{defn}

Definition~\ref{defn:admissible family} is a slight generalisation of
the original definition~\cite[Definition~5.1]{Loeh-Sauer19}, where
only families are considered that are closed under taking subgroups.

\begin{prop}\label{prop:ubacadm}
  Let $G$ be a group and let $\F$ be an intersection-closed family of
  subgroups of~$G$ that is uniformly boundedly acyclic
  and contains the trivial subgroup.  Then $\F$
  is strongly $H_b^*$-admissible.
\end{prop}
\begin{proof}
  We use Proposition~\ref{prop:stabacyc} and the fact that bounded
  cohomology can be computed using acyclic resolutions
  (Theorem~\ref{thm:fundamental lemma}).   
  Let $E_\F G$ be a model for the
  classifying space of~$G$ with respect to the family~$\F$. 
  We show that the chain complex~$C_b^*(E_\F G;\R)$ together with the
  canonical augmentation~$\R \to C_b^0(E_\F G;\R)$ is a boundedly
  acylic resolution of~$\R$ over~$G$: As $1 \in \F$, the space~$E_\F
  G$ is contractible (Theorem~\ref{thm:EFGex}). Therefore, $C_b^*(E_\F G;\R)$ is a resolution
  of~$\R$.

  Moreover, for each~$n \in \N$, the Banach $G$-module~$C_b^n(E_\F
  G;\R)$ is boundedly acyclic: By definition, $C_b^n(E_\F G;\R) =
  \linf(S,\R)$ with~$S \coloneqq \map(\Delta^n,E_\F G)$. In view of
  Proposition~\ref{prop:stabacyc}, it thus suffices to show that the
  stabilisers of the $G$-action on~$S$ lie in the uniformly boundedly
  acyclic collection~$\F$.  If $\sigma \colon \Delta^n \to E_\F G$ is
  a singular simplex, then $\sigma(\Delta^n)$ meets only finitely many
  cells of~$E_\F G$. Therefore, the stabiliser of~$\sigma$ is an
  intersection of finitely many elements of~$\F$ and thus lies
  in~$\F$.

  Let~$f\colon EG\to \EFG$ be the canonical (up to~$G$-homotopy) $G$-map.
  Then the fundamental lemma for boundedly acyclic resolutions
  (Theorem~\ref{thm:fundamental lemma}) shows that the cochain map
  \[ C_b^*(f;\R)^G \colon C_b^*(E_\F G;\R)^G \to C_b^*(E G;\R)^G 
  \]
  induces an isomorphism in bounded cohomology.
\end{proof}

\begin{cor}[{\cite[Proposition 5.2]{Loeh-Sauer19}}]\label{cor:AME admissible}
  Every intersection-closed family~$\F$ that consists of amenable
  groups and contains the trivial subgroup is
  strongly~$H^*_b$-ad\-mis\-si\-ble.
  \hfill\qedsymbol
\end{cor}


\subsection{A relative vanishing theorem}

In this section, we prove the relative vanishing theorem by making use of
the relative equivariant setting developed in the previous
sections. This result extends the classical vanishing theorem by
Ivanov~\cite{ivanov, ivanov_bac_covers} to the relative
setting. By now there are several alternative proofs of
Ivanov's result~\cite{FM:Grom, Loeh-Sauer19, Raptis_vanishing};
moreover, in the case of aspherical manifolds, the vanishing of
simplicial volume can also be obtained directly through the amenable
reduction lemma~\cite{alpertkatz, FM:Grom, loehmoraschinisauer}.  We will follow
the approach via classifying spaces by
L\"oh and Sauer~\cite{Loeh-Sauer19}.

First, we prove a vanishing result for the comparison map in terms of
the relative generalised LS-category (Definition~\ref{defn:factorisation}).

\begin{setup}[CW-pair with a family of subgroups]\label{setup:relcw}
  Let $(X,A)$ be a CW-pair with $X$ path-connected and
  let $A$ have only finitely many connected components.
  \begin{itemize}
  \item We suppose that the inclusion of every component of~$A$
    into~$X$ is $\pi_1$-injective and we let $(G,\calH = (H_i)_{i \in
      I})$ be a fundamental group pair for~$(X,A)$
    (Example~\ref{ex:fundamental group pair});
  \item Let $\F$ be a family of subgroups of~$G$ 
  that contains the trivial subgroup;
  \item For every~$i\in I$, let~$\F|_{H_i}$ be the restricted family of subgroups
  of~$H_i$ (Example~\ref{ex:family:subgroups}~\ref{item:restricted family});
  \item Let $p \colon \widetilde X \to X$ be the universal covering of~$X$.
  \end{itemize}
\end{setup}

\begin{lem}\label{lem:vanishing factorisation}
  In the situation of Setup~\ref{setup:relcw},
  suppose that the family~$\F$ of subgroups of~$G$ is strongly~$H^*_b$-admissible
  and that for every~$i\in I$ the family~$\F|_{H_i}$ of subgroups of~$H_i$ is
  strongly $H^*_b$-admissible.

  Then the canonical map~$f\colon(\widetilde{X},p^{-1}(A))\to \EFGH$
  of~$(G,\calH)$-CW-pairs
  induces an isomorphism in equivariant bounded cohomology:
  \[
  H^*_{G,b}(f)\colon
  H^*_{G,b}\bigl(\EFGH;\IR\bigr)
  \xrightarrow{\cong}
  H^*_{G,b}\bigl(\widetilde{X},p^{-1}(A);\IR\bigr) \,.
  \]
  In particular, the comparison map $\comp^k_{(X,A)}\colon
  H^k_b(X,A;\IR)\to H^k(X,A;\IR)$ vanishes in all degrees~$k\ge
  \cat_\F(X,A)$.
  \begin{proof}
    By Example~\ref{ex:universal covering pair},
  $(\widetilde{X},p^{-1}(A))$ is a~$(G, \calH)$-CW-pair with isotropy
  in the trivial family~$\TR$. Thus, since $\F$ contains the trivial subgroup,
  the $G$-map~$f$ can be factored (up to~$G$-homotopy) as
  $$\bigr(\widetilde{X},p^{-1}(A)\bigl) \xrightarrow{f_1} E_{\TR}(G, \calH) \xrightarrow{f_2} E_{\F}(G, \calH) \,.$$
  Using the long exact sequence for pairs in equivariant bounded cohomology
  together with the induction isomorphism (Lemma~\ref{lem:induction isomorphism}),
  we have the following:
  The induced map $H^*_{G,b}(f_1)$ is a (not necessarily isometric)
  isomorphism by the mapping theorem~\cite{vbc} and the five lemma.
  The map $H^*_{G,b}(f_2)$ is an isomorphism by the five lemma and
  the assumption that the families~$\F$ and~$\F|_{H_i}$ are
  strongly~$H^*_b$-admissible.	
  Together, this yields the desired isomorphism.
  
  Moreover, by naturality of the comparison map we have
  a commutative diagram, where we omit the trivial~$\IR$-coefficients:
  \[\begin{tikzcd}
  	H^k_{G,b}\bigr(\EFGH\bigl)\ar{rr}{H^k_{G,b}(f)}[swap]{\cong}\ar{d}[swap]{\comp^k_{G,\EFGH}}
	&& H^k_{G,b}\bigr(\widetilde{X},p^{-1}(A)\bigl)\ar{d}[swap]{\comp^k_{G,(\widetilde{X},p^{-1}(A))}}
	& H^k_b(X,A)\ar{l}{\cong}[swap]{H^*_b(p)}\ar{d}{\comp^k_{(X,A)}} \\
	H^k_G\bigr(\EFGH\bigl)\ar{rr}[swap]{H^k_G(f)}
	&& H^k_G\bigr(\widetilde{X},p^{-1}(A)\bigl)
	& H^k(X,A)\ar{l}{H^*(p)}[swap]{\cong}\,.
  \end{tikzcd}\]
  The claim follows, since the map~$H^k_G(f)$ is trivial in all 
  degrees~$k\ge \cat_\F(X,A)$ (Remark~\ref{rem:three defs}).
  \end{proof}
\end{lem}

Second, we use the relation between relative category and relative
multiplicity to derive a vanishing theorem via relative open covers:

\begin{setup}[Relative $\F$-cover and its nerves]\label{setup:relcover}
  In the situation of Setup~\ref{setup:relcw}, we
  additionally consider a relative $\F$-cover~$\calU$
  of~$(X,A)$ (Definition~\ref{defn:relative cover}).
  We will use the notions of
  relative multiplicity (Definition~\ref{def:relative:mult})
  and \admissibility\ (Definition~\ref{defn:relative cover}) of relative covers.
  Moreover, we fix the following notation:
  \begin{itemize}
  \item Let $\widetilde{\calU}$ be the lifted $G$-invariant
    cover of $\widetilde{X}$ (Example~\ref{ex:covers});
  \item Let 
    $
    \nu \colon\bigl(X,A\bigr)\to \bigl(|N(\calU)|, |N_A(\calU)|\bigr)
    $
    be a nerve map;
  \item Let
    $
    \widetilde{\nu} \colon
    \bigl(\widetilde{X},p^{-1}(A)\bigr)\to
    \bigl(|N(\widetilde{\calU})|, |N_{p^{-1}(A)}(\widetilde{\calU})|\bigr)
    $ 
    be a corresponding nerve map of $(G,\calH)$-CW-pairs
    (Proposition~\ref{prop:cover downstairs}).
  \end{itemize}
\end{setup}

The following theorem is a relative version of the
vanishing theorem for strongly $H_b^*$-admissible
families~\cite[Theorem~5.3]{Loeh-Sauer19}.

\begin{thm}[Relative vanishing theorem]\label{thm:relative vanishing theorem}
  In the situation of Setup~\ref{setup:relcw} and Setup~\ref{setup:relcover},
  suppose that the family~$\F$ of subgroups of~$G$ is
  strongly~$H^*_b$-admissible and that for all~$i\in I$ the
  family~$\F|_{H_i}$ of subgroups of~$H_i$ is
  strongly~$H^*_b$-admissible.
	
  Then the comparison map $\comp^*_G$ 
  factors through the equivariant nerve map $\widetilde{\nu}$:
  \[\begin{tikzcd}
  H^*_{G,b}(\widetilde{X},p^{-1}(A);\IR)\ar{r}{\comp^*_G}\ar[dashed]{dr} & H^*_G(\widetilde{X},p^{-1}(A);\IR) \\
  & H^*_G(|N(\widetilde{\calU})|,|N_{p^{-1}(A)}(\widetilde{\calU})|;\IR)\ar{u}[swap]{H_G^*(\widetilde{\nu};\IR)} \,.
  \end{tikzcd}\]
  In particular, the following hold:
  \begin{enumerate}[label=\enum]
  \item\label{item:relative vanishing theorem} 
	If $\calU$ is \admissible,
	then the comparison map $$\comp^k\colon H^k_b(X,A;\IR)\to
    H^k(X,A;\IR)$$ vanishes in all degrees~$k\ge \mult_A(\calU)$; 
    \item\label{item:relative nerve theorem} If $\calU$ is convex,
    then
    the comparison map $\comp^*$ factors through the nerve map~$\nu$:
    \[\begin{tikzcd}
    H^*_b(X,A;\IR)\ar{r}{\comp^*}\ar[dashed]{dr} & H^*(X,A;\IR) \\
    & H^*(|N(\calU)|,|N_A(\calU)|;\IR)\ar{u}[swap]{H^*(\nu;\IR)} \,.
    \end{tikzcd}\]
  \end{enumerate}
\end{thm}

\begin{proof}
  Let~$\EFGH$ be a model for the classifying space of~$(G,\calH)$ with respect to
  the family~$\F$ (Lemma~\ref{lem:existence classifying pair}).
  By Proposition~\ref{prop:cover downstairs},
  $(|N(\widetilde{\calU})|,|N_{p^{-1}(A)}(\widetilde{\calU})|)$ is a
  $(G,\calH)$-CW-pair with isotropy in $\F$. 
  By Example~\ref{ex:universal covering pair},
  $(\widetilde{X},p^{-1}(A))$ is a~$(G, \calH)$-CW-pair with isotropy
  in the trivial family~$\TR$.  
  The universal property of~$\EFGH$ yields that
  the canonical map~$f\colon (\widetilde{X},p^{-1}(A))\to \EFGH$ 
  of~$(G,\calH)$-CW-pairs factors
  (up to~$G$-homotopy) through the equivariant nerve map~$\widetilde{\nu}$:
  \[
  	\bigl(\widetilde{X},p^{-1}(A)\bigr)\xrightarrow{\widetilde{\nu}}
  	\bigl(|N(\widetilde{\calU})|,|N_{p^{-1}(A)}(\widetilde{\calU})|\bigr)\xrightarrow{\varphi} \EFGH \,.
  \]
  Hence, we have the following commutative diagram, where the trivial
  coefficient module~$\IR$ is omitted from the notation.
  \[\begin{tikzcd}
  H^*_b(X,A)\ar{rr}{\comp^*}\ar{d}{\cong}[swap]{H^*_b(p)}
  && H^*(X,A)\ar{d}[swap]{H^*(p)}{\cong}
  & H^*\bigl(|N(\calU)|,|N_A(\calU)|\bigr)\ar{l}[swap]{H^*(\nu)}\ar{d}{H^*(|N(p)|)}
  \\
  H^*_{G,b}\bigl(\widetilde{X},p^{-1}(A)\bigr)\ar{rr}{\comp^*_G}
  && H^*_G\bigl(\widetilde{X},p^{-1}(A)\bigr)
  & H^*_G\bigr(|N(\widetilde{\calU})|,|N_{p^{-1}(A)}(\widetilde{\calU})|\bigl)\ar{l}[swap]{H^*_G(\widetilde{\nu})} \\
  H^*_{G,b}\bigl(\EFGH\bigr)\ar{u}{H^*_{G,b}(f)}[swap]{\cong}\ar{rr}{\comp^*_{G,\EFGH}}
  && H^*_G\bigl(\EFGH\bigr)\ar{u}{H^*_G(f)}\ar{ru}[swap]{H^*_G(\varphi)} 
  \end{tikzcd}\]
  The map~$H^*_{G,b}(f)$ is an isomorphism by Lemma~\ref{lem:vanishing factorisation}.
  Therefore, the desired factorisation of the
  comparison map $\comp^*_G$ is given by
  \[
  \comp^*_G=H^*_G(\widetilde{\nu})\circ H^*_G(\varphi)\circ \comp^*_{G,\EFGH}\circ H^*_{G,b}(f)^{-1} \,.
  \]
  
  \ref{item:relative vanishing theorem}
  If $\calU$ is \admissible,  
  then we have~$\cat_\F(X,A)\le \mult_A(\calU)$ (Lemma~\ref{lem:factorisation multiplicity})
  and we conclude by Lemma~\ref{lem:vanishing factorisation}.
  More explicitly, let $n=\mult_A(\calU)$. Then, by
  Proposition~\ref{prop:cover downstairs}~\ref{item:dim mult}, we have
  $\dim(|N(\widetilde{\calU})|,|N_{p^{-1}(A)}(\widetilde{\calU})|) =
  n -1$ and hence
  $H^k_G(|N(\widetilde{\calU})|,|N_{p^{-1}(A)}(\widetilde{\calU})|)=0$
  for $k\ge n$. This shows that $\comp^k$ vanishes in every degree~$k
  \geq n$.
  
  \ref{item:relative nerve theorem}
  If $\calU$ is convex,
  then the map
  $H^*(|N(p)|)$ is an isomorphism by Proposition~\ref{prop:cover
    downstairs}~\ref{item:convex}.  Hence the desired factorisation of the
  comparison map $\comp^*$ is given by
  \[
  \comp^*=H^*(\nu) \circ H^*(|N(p)|)^{-1}\circ H^*_G(\varphi) \circ \comp^*_{G,\EFGH}\circ  H^*_{G,b}(f)^{-1}\circ H^*_b(p) \,.
  \]
  This proves the statement.
\end{proof}

\begin{rem}[$H^*_b$-admissible families]
  A family~$\F$ of subgroups of~$G$ containing the trivial subgroup is
  called (not necessarily strongly) \emph{$H^*_b$-admissible} if the
  canonical~$G$-map~$EG\to \EFG$ induces a surjective map in
  equivariant bounded cohomology in all
  degrees~\cite[Definition~5.1]{Loeh-Sauer19}.
	
  The conclusions of Theorem~\ref{thm:relative vanishing theorem} hold
  more generally, if in the situation of
  Setup~\ref{setup:relcover} the family~$\F$ of
  subgroups of~$G$ is only assumed to be (not necessarily strongly)
  $H^*_b$-admissible, while the families~$\F|_{H_i}$ of subgroups
  of~$H_i$ are still assumed to be strongly~$H^*_b$-admissible.
  Indeed, similarly to Lemma~\ref{lem:vanishing factorisation},
  in this case it follows from the four lemma for epimorphisms
  that the canonical~$G$-map~$(\widetilde{X},p^{-1}(A))\to \EFGH$ induces a
  surjective map in equivariant bounded cohomology in all degrees.
  This suffices to carry out the above proof of
  Theorem~\ref{thm:relative vanishing theorem}.
\end{rem}

As an application of the previous theorem, we can deduce the vanishing
theorem for uniformly boundedly acyclic open covers
(Definition~\ref{defn:UBAc cover}), which complements a recent result
by Ivanov~\cite{ivanov_bac_covers}.

\begin{cor}[Relative vanishing theorem for uniformly boundedly acyclic covers]
\label{cor:vanishing UBAc}
	In the situation of Setup~\ref{setup:relcw} and Setup~\ref{setup:relcover},
	suppose that the relative cover~$\calU$ of~$(X,A)$ is
	uniformly boundedly acyclic, viewed as an open cover of~$X$.
	
	Then all statements in Theorem~\ref{thm:relative vanishing theorem} hold.
	
	\begin{proof}
		We take~$\F$ to be the family~$\F\spann{\calU}\cup \{1\}$ of subgroups 
		of~$G$.
		Since the cover~$\calU$ is uniformly boundedly acyclic,
		the family~$\F$ is uniformly boundedly acyclic and hence
		strongly $H^*_b$-admissible (Proposition~\ref{prop:ubacadm}).
		As subsets of uniformly boundedly acyclic sets of groups are again
		uniformly boundedly acyclic, the restricted families~$\F|_{H_i}$ of 
		subgroups of~$H_i$ are strongly $H^*_b$-admissible for every~$i\in I$.
		Thus, Theorem~\ref{thm:relative vanishing theorem} applies and
		yields the thesis.
	\end{proof}
\end{cor}

\begin{rem}[$\ell^1$-Homology]
  In view of \emph{strong} $H_b^*$-admissibility, the analogues of
  Theorem~\ref{thm:relative vanishing theorem} and 
  Corollary~\ref{cor:vanishing UBAc} for $\ell^1$-homology also
  hold. One can simply argue through duality as in the absolute
  case~\cite[Section~6]{Loeh-Sauer19}.
\end{rem}

In particular, Theorem~\ref{thm:relative vanishing theorem} gives a
relative vanishing theorem in the presence of ``small" relative amenable
\category\ (Definition~\ref{defn:relative category}).

\begin{cor}[Relative vanishing theorem for amenable covers]\label{cor:rel:van:thm:ame}
  Let $(X,A)$ be a CW-pair with path-connected ambient space~$X$.  Assume that~$A$
  consists of finitely many connected components, each of which
  is~$\pi_1$-injective in $X$.  Then the comparison map
  \[
  \comp^k_{(X,A)}\colon H^k_b(X,A;\IR)\to H^k(X,A;\IR)
  \] 	
  vanishes in all degrees~$k \geq \relcat_\AME(X,A)$.
  
  In particular, if $(M,\partial M)$ is an oriented compact connected
  triangulable manifold with (possibly empty) $\pi_1$-injective
  boundary and $\relcat_{\AME}(M, \partial M) \leq \dim(M)$, then the relative
  simplicial volume~$\|M,\partial M\|$ vanishes.
\end{cor}

\begin{proof}
  The claim on the comparison map follows from the
  strong~$H^*_b$-admissibility of the family~$\AME$
  (Corollary~\ref{cor:AME admissible}).
  
  Because $(M,\partial M)$ is triangulable, both $M$ and $\partial M$
  admit compatible triangulations; in particular, we can
  view~$(M,\partial M)$ as a CW-pair.  Therefore, the statement on
  relative simplicial volume follows from the duality principle
  (Proposition~\ref{duality:principle})
\end{proof}

\begin{rem}[Optimality of assumptions]
  Let $\Sigma$ denote the surface of genus~$1$ with one boundary
  component.  Since the interior of~$\Sigma$ admits a complete
  hyperbolic metric with finite volume, we know that $\sv{\Sigma,
    \partial\Sigma} > 0$ (Example~\ref{ex:relative:sv}.1).
  This shows that~$\relcat_\AME(\Sigma,\partial\Sigma)=3$.  However, when
  either one of the conditions
  ~\ref{item:RC1},~\ref{item:RC2}, or \admissibility\ on the open cover
  is dropped, then it is not difficult to
  construct amenable covers of~$(\Sigma,\partial\Sigma)$ with relative
  multiplicity at most~$2$.  In this sense, our set of
  assumptions in Corollary~\ref{cor:rel:van:thm:ame} is optimal.
\end{rem}

\subsection{Amenable covers with small multiplicity on the boundary}
\label{sec:Gromov relative vanishing}

We compare Corollary~\ref{cor:rel:van:thm:ame} with
existing results in the literature~\cite[Section~3.3]{LMR}.  
The main available relative vanishing result is the following,
which is based on Gromov's
vanishing theorem for non-compact manifolds~\cite{vbc}\cite[Corollary~11]{FM:Grom}: 
   
\begin{thm}[{\cite[Theorem~3.13]{LMR}}]\label{thm:rel:van:thm:LMR}
  Let $(M, \partial M)$ be an oriented compact connected $n$-manifold
  with non-empty boundary.  Let $\calU$ be an open cover of~$M$
  and let $\calU |_{\partial M}$ denote the restriction of~$\calU$ 
  to~$\partial M$, i.e., $\calU |_{\partial M} \coloneqq \{U \cap \partial M \mid U \in \calU\}$.
  Suppose that the following hold:
  \begin{enumerate}[label=\enum]
  \item\label{item:Gromov1} $\mult(\calU) \leq n$;
  \item\label{item:Gromov2} $\mult(\calU |_{\partial M}) \leq n-1$;
  \item\label{item:Gromov3} The open covers~$\calU$ of~$M$ and $\calU |_{\partial
    M}$ of~$\partial M$ are amenable.
  \end{enumerate}
  Then  $\sv{M, \partial M} = 0$.
\end{thm}

	We emphasise that (contrary to our convention) 
	here the restricted cover~$\calU|_{\partial M}$ of~$\partial M$
	may consist of disconnected subsets.

  Since this result neither assumes that
  the manifold~$M$ is triangulable, nor that the boundary
  inclusion is $\pi_1$-injective (in the sense of
  Section~\ref{subsec:convs}), while our
  Corollary~\ref{cor:rel:van:thm:ame} does so,
  we cannot recover Theorem~\ref{thm:rel:van:thm:LMR} in full generality.
  However, in the special situation of triangulable manifolds
  with $\pi_1$-injective 
  boundary, we can provide a simplified proof that
  does not make use of Gromov's theory of diffusion of chains:
  
\begin{proof}[Proof of Theorem~\ref{thm:rel:van:thm:LMR} for triangulable manifolds
      with $\pi_1$-injective boundary]\hfil 
  We show that there exists a
  \admissible\ relative $\AME$-cover~$\calV$ of~$(M, \partial M)$ such
  that~$\mult_{\partial M}(\calV) \leq n = \dim(M)$.  This implies that
  $\relcat_{\AME}(M, \partial M) \leq n$ and thus by
  Corollary~\ref{cor:rel:van:thm:ame} the thesis.
  
  Let~$m\coloneqq \{\mult(\calU),\mult(\calU|_{\partial M})+1\}$.
  By conditions~\ref{item:Gromov1} and~\ref{item:Gromov2}, we have~$m\le n$.
  From~$\calU$ we will construct a new cover~$\calV$ of~$(M,\partial M)$ that is
  a weakly convex relative $\AME$-cover with $\mult_{\partial M}(\calV)\le m$.
  To this end, we follow a classical strategy for modifying open
  covers (while controlling the multiplicity) in the case of compact
  manifolds with boundary~\cite[proof of
    Theorem~11.2.3]{FM:Grom}\cite[proof of
    Theorem~5.3]{Loh-Sauer-degree}.

  By compactness, we may assume that $\calU$ is finite, say $\calU = \{U_1, \dots, U_k\}$.
  Since $\partial M$ is collared in $M$, we have the following identification:
  \[
  (M, \partial M)
  \cong (M', \partial M') \coloneqq \bigl(M \cup_{\partial M \cong (\partial M \times \{0\})}
    (\partial M \times [0, 1]), \partial M \times \{1\}\bigr) \,.
  \]
  Let $\varepsilon \coloneqq 1 \slash 3(k+1)$ and $t_i \coloneqq i \slash (k+1)$.
  Moreover, for $i \in \{1, \dots, k\}$, let $\calU(i)$ denote the set of connected 
  components of $U_i \cap \partial M$. We set for every $i \in \{1, \dots, k\}$
  such that $U_i \cap \partial M \neq \emptyset$
  \[
  U_i' \coloneqq U_i \cup \bigl( (U_i \cap \partial M) \times [0, t_i + \varepsilon) \bigr) \subset M'
  \]
  and
  \[
  \calU'(i) \coloneqq \bigl\{U \times (t_i - \varepsilon, 1] \bigm| U \in \calU(i)\bigr\} \,.
  \]  
  When $U_i \cap \partial M = \emptyset$, we just set $U_i' \coloneqq U_i$ and $\mathcal{U}'(i) \coloneqq \emptyset$.
  This produces a new open amenable cover 
  \[
  \calU' \coloneqq \bigl\{U_1', \dots, U_k'\bigr\} \cup \calU'(1) \cup \dots \cup \calU'(k) 
  \]
  of $M'$.
  
  Finally, we obtain $\calV$ from $\calU'$ by discarding all sets of the form $U \times (t_i - \varepsilon, 1]$
  with $U \in \calU(j)$ for some $j > i$. Then $\calV$ is an amenable open cover
  of $M'$ and a straightforward case analysis shows that $\mult(\calV) \leq m$.
  In particular, this implies that 
  $\mult_{\partial M'}(\calV) \leq \mult(\calV) \leq m \leq n$
  by assumptions~\ref{item:Gromov1} and~\ref{item:Gromov2}.
     
  We are left to show that $\calV$ satisfies all the conditions in
  Definition~\ref{defn:relative cover}.  The key observation is the
  following: Only sets $V \in \calV$ of the form $V = U \times (t_i - \varepsilon, 1]$
  with $U \in \calU(i)$ intersect $\partial M'$. 
  In particular, if $V \in \calV$
  satisfies~$V \cap \partial M' \neq \emptyset$, then $V \cap \partial
  M'$ is connected. Thus, \ref{item:RC1} is satisfied.
  Moreover, for the same reason every $V \in \calV$ 
  with $V \cap \partial M' \neq \emptyset$ 
  deformation retracts onto $V \cap \partial M'$,
  whence~\ref{item:RC2} holds.  
  It remains to show that the relative cover $\calV$
  is weakly convex. If $V_1,
  \dots, V_j \in \calV$ satisfy $V_1 \cap \dots \cap V_j \cap
  \partial M' \neq \emptyset$, then  $V_1 \cap
  \dots \cap V_j$ is of the form
  $(V_1 \cap \dots \cap V_j \cap \partial M') \times (r,
  1]$ for some $r \in (0, 1)$.
  In particular, each component of $V_1 \cap \dots \cap V_j$
  intersects $\partial M'$. Hence, $\calV$ is weakly convex.
  
  We conclude that $\calV$ is a \admissible\ 
    relative $\AME$-cover of~$(M',\partial M')$ with $\mult_{\partial M'}
    (\calV) \leq n$. Using the identification of~$(M',\partial M')$
    with~$(M, \partial M)$, we get the thesis.
\end{proof}

We conclude this section by showing that for general CW-pairs~$(X, A)$
such that the inclusion of~$A$ into~$X$ is $\pi_1$-injective the
hypotheses of Corollary~\ref{cor:rel:van:thm:ame} are weaker than the
ones of Theorem~\ref{thm:rel:van:thm:LMR}:

\begin{ex}\label{ex:two:cover:conditions:are:diff}
  Let $n \geq 3$. Let $M$ be an oriented closed connected hyperbolic
  $(n-1)$-manifold and denote $I=[0,1]$. 
  We consider the CW-pair $(M \times I, M \times \{1\})$.  
  Since $M \times \{1\} \cong M$ is hyperbolic, we have
  $\sv{M \times \{1\}} > 0$~\cite{Thurston, vbc}.  Hence, $M \times
  \{1\}$ cannot admit an open amenable cover with multiplicity at most~$n-1$.
  This shows that there is no open amenable cover of $M$
  whose restriction to $M \times \{1\}$ is both amenable and with
  multiplicity at most $n-1$. 

  On the other hand, it is easy to construct a \admissible\ relative
  $\AME$-cover~$\calV$ of~$(M \times I,M\times \{1\})$ with $\mult_{M
    \times \{1\}}(\calV) \leq n$.  Let $\calU$ be the open star cover
  of~$M \times \{1\}$ and let $\calV$ be the cover of $M \times I$
  defined as follows:
  $$
  \calV \coloneqq \{U \times I \mid U \in \calU\} \,.
  $$
  Since by construction $\calV$ consists of contractible sets,
  $\calV$ is an amenable open cover.
   Moreover, each member of $\calV$ intersects $M \times \{1\}$ in a
  contractible, whence connected, set. 
  The same argument also applies to multiple intersections. Finally, since each element $V$ in
  $\calV$ is a product $U \times I$ with $U \in \calU$, it
  follows that $V$ retracts by deformation onto $V \cap (M
  \times \{1\})$.  This shows that $\calV$ is a \admissible\ relative
  $\AME$-cover of $(M \times I, M \times \{1\})$. 
  
  Since by construction the
  relative multiplicity of $\calV$ is zero, we have obtained our
  desired cover.
\end{ex}

We do not know whether the previous example can be improved to the situation
of compact manifolds with $\pi_1$-injective boundary.

\section{A vanishing theorem for relative $\ell^2$-Betti numbers}

We prove a vanishing theorem for the relative $\ell^2$-Betti numbers
of aspherical CW-pairs with small relative amenable \category\ using
equivariant nerve pairs.  In the absolute case for~$\ell^2$-Betti
numbers, more sophisticated arguments involving nerves have previously
been used by Sauer~\cite{Sauer09}.  For further background on
$\ell^2$-Bet\-ti numbers we refer to the
literature~\cite{Lueck02,Kammeyer19}.  The results of this section are
not used in the rest of the article.

We use the following (non-standard) notation:

\begin{defn}[$\ell^2$-Homology and $\ell^2$-Betti numbers]
  Let $Y$ be a $G$-space. The
  \emph{$\ell^2$-homology}~$H^{(2)}_*(G\actson Y)$ is defined as the
  singular homology of~$Y$ with twisted coefficients in the group
  von Neumann algebra~$\calN G$, that is the $\calN G$-module
  \[
  H^{(2)}_*(G\actson Y)\coloneqq H_*^G(Y;\calN G) \,.
  \]
  The \emph{$n$-th $\ell^2$-Betti number}~$b^{(2)}_n(G\actson Y)\in
  \IR_{\ge 0} \cup\{\infty\}$ is
  \[
  b^{(2)}_n(G\actson Y)\coloneqq \dim_{\calN G}\bigl(H^{(2)}_n(G\actson Y)\bigr) \,,
  \]
  where $\dim_{\calN G}$ is the von Neumann dimension function.  For a
  pair~$(Y,B)$ of $G$-spaces, one similarly defines
  $H^{(2)}_*(G\actson (Y,B))$ and~$b^{(2)}_*(G\actson (Y,B))$.
	
  The $\ell^2$-Betti numbers~$b^{(2)}_*(G)$ of a group~$G$ are
  defined as
  \[
  b^{(2)}_*(G)\coloneqq b^{(2)}_*(G\actson EG) \,.
  \]
  We say that~$G$ is \emph{$\ell^2$-acyclic} if~$b^{(2)}_k(G)=0$ for
  all~$k>0$.
\end{defn}

For example, amenable groups are $\ell^2$-acyclic~\cite[Corollary~6.75]{Lueck02}.
The following shows that
$\ell^2$-Betti numbers can be computed using classifying spaces for
families consisting of $\ell^2$-acyclic subgroups.

\begin{prop}[{\cite[Theorem~4.14]{Kammeyer19}}]\label{prop:Borel}
  Let $G$ be a group and $\F$ be a (con\-ju\-ga\-tion-closed) family of
  subgroups of~$G$ that consists of $\ell^2$-acyclic groups and
  contains the trivial subgroup.  Then the canonical~$G$-map~$EG\to
  \EFG$ induces a dimension-isomorphism in $\ell^2$-homology
  \[
  H^{(2)}_*(G\actson EG)\xrightarrow{\cong_{\dim}} H^{(2)}_*(G\actson \EFG) \,.
  \]
  In particular, $b^{(2)}_*(G)=b^{(2)}_*(G\actson \EFG)$.
\end{prop}
\begin{proof}
  Since $\F$ contains the trivial subgroup, $EG \times \EFG$ equipped
  with the diagonal~$G$-action is a model for~$EG$. The projection~$EG\times \EFG\to \EFG$
  onto the second factor induces a dimension-isomorphism
  \[ 
  H^{(2)}_*(G\actson EG)
  \cong 
  H^{(2)}_*(G\actson EG \times \EFG) \xrightarrow{\cong_{\dim}}
  H^{(2)}_*(G\actson \EFG) \,,
  \]
  because all members of~$\F$ are $\ell^2$-acyclic~\cite[proof of
    Theorem~6.54~(2)]{Lueck02}.
\end{proof}

For the vanishing theorem, we consider the following situation:
Let~$(X,A)$ be a CW-pair with~$X$ path-connected.
Suppose that~$A$ has only finitely many connected components and
let~$A=\coprod_{i\in I} A_i$ be a decomposition into connected components.
Assume that each~$A_i$ is~$\pi_1$-injective in~$X$ and
let~$(G,\calH)$ be a fundamental group pair for~$(X,A)$
(Example~\ref{ex:fundamental group pair}).
Let~$\F$ be a family of subgroups of~$G$ that contains the trivial subgroup.
Denote by~$p\colon \widetilde{X}\to X$ the universal covering.

Moreover, let $\calU$ be a relative~$\F$-cover of~$(X,A)$ 
(Definition~\ref{defn:relative cover}) with relative multiplicity~$\mult_A(\calU)$ 
(Definition~\ref{def:relative:mult}).
Let~$\widetilde{\calU}$ be the lifted $G$-invariant cover of~$\widetilde{X}$
(Example~\ref{ex:covers}) and
$(|N(\widetilde{\calU})|,|N_{p^{-1}(A)}(\widetilde{\calU})|)$ be its
equivariant nerve pair (Proposition~\ref{prop:cover downstairs}).
We will also use the notion of weakly convex relative covers
(Definition~\ref{defn:relative cover}).

\begin{thm}[Relative vanishing theorem for $\ell^2$-Betti numbers]\label{thm:l2rel}
  In the above situation,
  if $X$ and all the $(A_i)_{i\in I}$ are aspherical and $\F$ consists of
  $\ell^2$-acyclic groups, then we have
	\[
  b^{(2)}_*\bigl(G\actson (\widetilde{X},p^{-1}(A))\bigr)
  \le
  b^{(2)}_*\bigl(G\actson (|N(\widetilde{\calU})|,|N_{p^{-1}(A)}(\widetilde{\calU})|)\bigr) \,.
	\]
	In particular, if~$\calU$ is \admissible,
	then $$b^{(2)}_k\bigl(G\actson (\widetilde{X},p^{-1}(A))\bigr)=0$$ for all~$k\ge \mult_A(\calU)$.
\end{thm}
\begin{proof}
  By Proposition~\ref{prop:cover downstairs}, the equivariant nerve pair
  $(|N(\widetilde{\calU})|,|N_{p^{-1}(A)}(\widetilde{\calU})|)$ is a~
  $(G,\calH)$-CW-pair with isotropy in $\F$. 
  By the universal property of the classifying space~$\EFGH$,
  the canonical
  $G$-map $f\colon (\widetilde{X},p^{-1}(A))\to \EFGH$ factors 
  (up to~$G$-homotopy) through
  the equivariant nerve map~$\widetilde{\nu}$:
  \[
  \bigl(\widetilde{X},p^{-1}(A)\bigr)\xrightarrow{\widetilde{\nu}}
  \bigl(|N(\widetilde{\calU})|,|N_{p^{-1}(A)}(\widetilde{\calU})|\bigr)
  \to \EFGH \,.
  \]
  Since $X$ and all the $(A_i)_{i\in I}$ are aspherical by assumption,
  Proposition~\ref{prop:Borel} and a five lemma for
  dimension-isomorphisms~\cite[Section~2]{Sauergroupoids}
  imply that $f$ induces a dimension-isomorphism
  \[
  H^{(2)}_*\bigl(G\actson (\widetilde{X},p^{-1}(A))\bigr)
  \xrightarrow{\cong_{\dim}}
  H^{(2)}_*\bigl(G\actson \EFGH\bigr) \,.
  \]
  Thus, the above factorisation of the~$G$-map~$f$ shows 
  \[
  b^{(2)}_*\bigl(G\actson (\widetilde{X},p^{-1}(A))\bigr)
  \le b^{(2)}_*\bigl(G\actson (|N(\widetilde{\calU})|,|N_{p^{-1}(A)}(\widetilde{\calU})|)\bigr) \,,
  \]
  as claimed.
  
  To conclude the vanishing result, suppose that the relative cover~$\calU$ is
  \admissible\ with $\mult_A(\calU)=n$.
  Then we have
  $\dim(|N(\widetilde{\calU})|,|N_{p^{-1}(A)}(\widetilde{\calU})|)=n-1$ 
  by Proposition~\ref{prop:cover downstairs}~\ref{item:dim mult}
  and hence
  $b^{(2)}_k(G\actson (|N(\widetilde{\calU})|,|N_{p^{-1}(A)}(\widetilde{\calU})|))=0$
  for all degrees~$k\ge n$. 
\end{proof}

In the absolute case for the family~$\AME$, we recover a 
result of Sauer~\cite[Theorem~C]{Sauer09}:

\begin{cor}\label{cor:l2abs}
  Let $X$ be a path-connected aspherical CW-complex. Then we have
  $b^{(2)}_k(\pi_1(X)\actson \widetilde{X})=0$ for all~$k\ge
  \cat_\AME(X)$.
  \hfill\qedsymbol
\end{cor}

\section{Glueing estimates for relative simplicial volume}\label{sec:bacglue}

The classical glueing estimates for simplicial
volume~\cite{vbc,BBFIPP, Kuessner} require that the boundary
components used in the glueing have amenable fundamental
groups. Replacing amenability with bounded acyclicity and the uniform
boundary condition also leads to (albeit weaker) glueing estimates for
simplicial volume (Theorem~\ref{thm:bacglue}). The uniform boundary
condition also has been used for versions of simplicial volume where
the strong glueing formulae are unknown and no suitable dual theory is
available~\cite{fauserloeh_varubc,fauserfriedlloeh}.

\begin{thm}[Vanishing inheritance for boundedly acyclic glueings]\label{thm:bacglue}
  Let $n\ge 3$ and $(M_i,\partial M_i)_{i\in I}$ be a finite
  collection of oriented compact connected $n$-manifolds.  Assume that
  every connected component of every boundary component~$\partial M_i$
  has boundedly acyclic fundamental group.  Let $\calN$ be a set
  of~$\pi_1$-injective boundary components of the~$(M_i)_{i\in I}$ and let
  $(M,\partial M)$ be obtained from~$(M_i,\partial M_i)_{i\in I}$ by a
  pairwise glueing of the boundary components in~$\calN$ (along
  orientation-reversing homeomorphisms).
        
  If~$\calN$, viewed as a set of subsets of~$M$, is uniformly
  boundedly acyclic of order~$n$ in~$M$ (Definition~\ref{defn:UBAc set of
    subspaces}) then the following are equivalent:
  \begin{enumerate}[label=\enum]
  \item\label{i:M}
    We have~$\sv{M,\partial M} = 0$;
  \item\label{i:Mi}
    For all~$i \in I$, we have~$\sv{M_i,\partial M_i} = 0$.
  \end{enumerate}
\end{thm}

The implication \ref{i:Mi} $\Rightarrow$ \ref{i:M} is proved in
Section~\ref{subsec:ubcglue}. The implication \ref{i:M}
$\Rightarrow$ \ref{i:Mi} is established in
Section~\ref{subsec:lowerglue}.

\begin{rem}
  More generally, the conclusion of Theorem~\ref{thm:bacglue} holds
  for \emph{compatible} glueings~\cite{BBFIPP}, where the boundary
  components in~$\calN$ need not be~$\pi_1$-injective.
  On the other hand, it remains unclear whether the
  assumption of bounded acyclicity on the boundary components that are
  not in~$\calN$ can be dropped.
\end{rem}

\begin{rem}
  In the situation of Theorem~\ref{thm:bacglue}, if the collection of
  fundamental groups of all members of~$\calN$ is \emph{malnormal}
  in~$\pi_1(M)$, then the uniform bounded acyclicity condition is
  automatically satisfied as soon as all members of~$\calN$ have
  boundedly acyclic fundamental group.
\end{rem}

\subsection{Upper glueing estimates via the uniform boundary condition}\label{subsec:ubcglue}

We begin with the, easier, upper glueing estimate; this estimate works
over all normed rings and also gives rough estimates in the
non-vanishing case:

\begin{prop}\label{prop:ubcupperestimate}
  Let $R$ be a normed ring, let $K \in \R_{>0}$, let $I$ be a finite
  set, and let $(M_i,\partial M_i)_{i \in I}$ be a finite collection
  of oriented compact connected manifolds of the same
  dimension~$n$. Moreover, let $(M,\partial M)$ be obtained
  from~$(M_i,\partial M_i)_{i \in I}$ by a pairwise glueing (along
  orientation-reversing homeomorphisms) of a set~$\calN$ of boundary
  components such that $K$ is a $\UBC_{n-1}$-constant of~$C_*(\bigcup \calN;R)$.  Then
  \begin{align*}
    \sv{M,\partial M}_R
    & \leq \bigl( 1+ K \cdot (n+1) \bigr) \cdot \sum_{i \in I} \sv{M_i,\partial M_i}_R \,.
  \end{align*}
  In particular, if $\sv{M_i,\partial M_i}_R = 0$ for all~$i \in I$,
  then~$\sv{M,\partial M}_R = 0$.
\end{prop}
\begin{proof}
  This is the standard filling
  argument~\cite[Example~6.18]{loeh_phd}\cite[Remark~6.2]{BBFIPP},
  adapted to the general $\UBC$-setting; for the sake of completeness,
  we give the argument:

  For notational simplicity, we view the $(M_i)_{i\in I}$ as subspaces of the
  glued manifold~$M$. Moreover, we write $N \subset M$ for the
  (disjoint) union of the glueing loci. 
  Hence, $K$ is a
  $\UBC_{n-1}$-constant for~$N$.  Let $(z_i \in C_n(M_i;R))_{i \in I}$
  be a collection of relative fundamental cycles of~$(M_i,\partial
  M_i)_{i \in I}$. As we glue along orientation-reversing
  homeomorphisms, the chain
  \[ b \coloneqq \sum_{i \in I} \partial z_i|_{N} \in C_{n-1}(N ; R)
  \]
  is null-homologous. By $\UBC_{n-1}$ for~$C_*(N;R)$,
  there exists a chain~$c \in C_n(N;R)$ with
  \[ \partial c = b
  \qand
  |c|_1 \leq K \cdot |b|_1
  \leq K \cdot \sum_{i \in I} |\partial z_i|_1
  \leq K \cdot (n+1) \cdot \sum_{i \in I} |z_i|_1 \,.
  \]
  Then $z \coloneqq \sum_{i \in I} z_i - c \in C_n(M;R)$ is a relative cycle
  on~$(M,\partial M)$; checking the local contributions on the
  components~$(M_i,\partial M_i)$ shows that $z$ is a relative
  $R$-fundamental cycle of~$(M,\partial M)$. Therefore, we obtain
  \begin{align*}
    \sv{M,\partial M}_R
    & \leq |z|_1
      \leq \sum_{i \in I} |z_i|_1 + |c|_1
    \\
    & \leq \sum_{i \in I} |z_i|_1 + K \cdot (n+1) \cdot \sum_{i \in I} |z_i|_1 \,.
  \end{align*}
  Taking the infimum over all relative fundamental cycles~$(z_i)_{i \in I}$
  proves the claim.
\end{proof}

\begin{proof}[Proof of Theorem~\ref{thm:bacglue}, \ref{i:Mi} $\Rightarrow$ \ref{i:M}]
  All boundedly acyclic groups satisfy the uniform boundary condition
  in all degrees (Theorem~\ref{thm:ubcbanach}). As only finitely many
  components are involved, we also find a joint $\UBC_{n-1}$-constant
  for~$C_*(\bigcup\calN;\R)$.  Applying Proposition~\ref{prop:ubcupperestimate}
  thus proves the implication~\ref{i:Mi}~$\Rightarrow$~\ref{i:M}.
\end{proof}

In the same way, we also obtain the following estimate for the locally
finite simplicial volume~\cite{vbc,loeh_phd} for interiors of compact
manifolds with $\UBC$-boundary; similar results have been obtained
previously for amenable boundaries (or other more restrictive
conditions on the boundary) via the uniform boundary
condition~\cite{loehsauer_hilbert,loeh_phd} or bounded
cohomology~\cite{Kim-Kuessner15}.

\begin{prop}\label{prop:svlfgen}
  Let $R$ be a normed ring and let $M$ be an oriented connected
  compact $n$-manifold with boundary satisfying the following properties:
  \begin{itemize}
  \item We have~$\sv {\partial M}_R = 0$;
  \item 
    The boundary~$\partial M$ satisfies~$\UBC_{n-1}$ over~$R$; 
    let $K$ be a $\UBC_{n-1}$-constant for~$C_*(\partial M;R)$.
  \end{itemize}
  Then
  \[ \svlfR {M^\circ} \leq \bigl( K \cdot (n+1) + 1 \bigr) \cdot \sv{M,\partial M}_R \,.
  \]
\end{prop}
\begin{proof}
  We will follow the previously known $\UBC$-arguments: Let $c \in
  C_n(M;R)$ be a relative fundamental cycle of~$(M,\partial M)$ and
  let $\varepsilon \in \R_{>0}$.

  Because of $\sv{\partial M}_R = 0$, there exists a
  sequence~$(z_k)_{k \in \N}$ in~$C_{n-1}(\partial M;R)$ of
  fundamental cycles of~$\partial M$ with~$|z_k|_1 \leq \varepsilon
  \cdot 1/2^k$ for all~$k \in \N$.  From this sequence, we can
  construct a locally finite relative fundamental cycle~$z \in
  C_n^{\lfop}(\partial M \times \R_{\geq0};R)$ of the half-open
  cylinder~$\partial M \times \R_{\geq0}$ with $\partial z = -
  \partial c$ and
  \begin{align*}
    |z|_1
    & \leq K \cdot \bigl( |\partial c|_1 + |z_0|_1\bigr) + n \cdot |z_0|_1
    + \sum_{k=0}^\infty \bigl(K \cdot \bigl(|z_k|_1 + |z_{k+1}|_1 \bigr) + n \cdot |z_{k+1}|_1\bigr)
    \\
    & \leq K \cdot |\partial c|_1 + n \cdot \varepsilon \cdot 2
    + K \cdot 2 \cdot \varepsilon \cdot 2
    \\
    & \leq K \cdot (n+1) \cdot |c|_1 + 2 \cdot \varepsilon \cdot (n + 2 \cdot K)
  \end{align*}
  by $\UBC$-filling the differences between subsequent~$z_k$ with
  ``small'' chains and then spreading out the result over the
  half-open cylinder~\cite[proof of Theorem~6.1]{loeh_phd}. 
  Here we filled the ``cylinders" by using the canonical triangulation
  of $\Delta^{n-1} \times [0,1]$ into $n$~simplices of dimension~$n$.

  Moreover, $c+z$ is a locally finite fundamental cycle of~$M
  \cup_{\partial M} (\partial M \times \R_{\geq 0}) \cong M^\circ$ (via
  the topological collar theorem) and so
  \[ \svlfR{M^\circ}
  \leq |c+z|_1
  \leq |c|_1 + K \cdot (n+1) \cdot |c|_1 + 2 \cdot \varepsilon \cdot (n + 2 \cdot K) \,.
  \]
  Taking the infimum over~$\varepsilon \to 0$
  and then over all relative fundamental cycles~$c$ thus shows that
  \[ \svlfR{M^\circ}
  \leq \sv{M,\partial M}_R
  + K \cdot (n+1) \cdot \sv{M,\partial M}_R
  \,,
  \]
  as claimed.
\end{proof}

\begin{cor}\label{cor:svlfzero}
  Let $M$ be an oriented connected compact $n$-manifold with boundary
  satisfying the following properties:
  \begin{itemize}
  \item We have~$\sv{M,\partial M} = 0$;
  \item The boundary~$\partial M$ satisfies~$\UBC_{n-1}$.
  \end{itemize}
  Then $\svlf{M^\circ} = 0$.
\end{cor}
\begin{proof}
  Since $\sv{\partial M} \leq (n+1) \cdot \sv{M,\partial M}= 0$,
  we can apply Proposition~\ref{prop:svlfgen}.
\end{proof}

In particular, the conditions on the boundary in
Proposition~\ref{prop:svlfgen} (over~$\R$) and
Corollary~\ref{cor:svlfzero} are satisfied if the boundary 
has vanishing bounded cohomology in positive degrees.

\begin{rem}[Group-theoretic Dehn fillings]
  A classical application of upper glueing estimates are (generalised)
  Dehn fillings of manifolds~\cite{FujiwaraManning,BBFIPP}.
  The simplicial volume of \emph{group-theoretic} Dehn fillings was
  recently investigated~\cite{Petrosyan-Sun21}.
  In particular, the simplicial volume does not
  increase when performing a group-theoretic Dehn filling with
  resulting peripheral subgroups that are amenable and of small
  cohomological dimension~\cite[Theorem~6.3]{Petrosyan-Sun21}.  One
  obtains an analogous result for the vanishing behaviour of
  simplicial volume if amenability is replaced by bounded acyclicity.
\end{rem}

\subsection{Lower glueing estimates via bounded acyclicity}\label{subsec:lowerglue}

We prove the lower glueing estimate~\ref{i:M}
$\Rightarrow$~\ref{i:Mi} of Theorem~\ref{thm:bacglue}, adapting the
argument in the amenable case by Bucher, Burger, Frigerio, Iozzi,
Pagliantini, and Pozzetti~\cite{BBFIPP}.

In this section, all [bounded] cohomology groups are taken with trivial
coefficients in~$\IR$.

\begin{proof}[Proof of Theorem~\ref{thm:bacglue}, \ref{i:M} $\Rightarrow$ \ref{i:Mi}]
  We proceed by contraposition, i.e., we assume that one of the
  building blocks satisfies~$\sv{M_i,\partial M_i} > 0$ and show that
  $\sv{M,\partial M} > 0$.
  By the duality principle (Proposition~\ref{duality:principle}),
  it suffices to show that the comparison map~$H_b^n(M,\partial M) \to
  H^n(M,\partial M)$ is non-zero.
  
  Let~$N\subset M$ be the union of the glueing loci. 
  We consider the following diagram. Here, all unlabelled
  arrows are induced by inclusions (and direct sums) and the
  maps labelled by~$\ev$ are given by evaluation on the fundamental
  class.
  \[ \xymatrix{%
    H_b^n(M)
    \ar[dd]_-{(**)}
    &
    H_b^n(M,\partial M)
    \ar[l]_-{\cong}
    \ar[r]_-{\comp_{(M,\partial M)}^n}
    &
    H^n(M,\partial M)
    \ar[r]^-{\cong}_-{\ev}
    &
    \R
    \\
    &
    H_b^n(M,N \cup \partial M)
    \ar[u]^-{\cong}
    \ar[r]_*+<.5em>{_{\comp_{(M,N \cup \partial M)}^n}}
    \ar[d]_-{(***)}
    &
    H^n(M,N \cup \partial M)
    \ar[u]
    \ar[d]
    \\
    \bigoplus_{i \in I} H_b^n(M_i)
    &
    \bigoplus_{i \in I} H_b^n(M_i,\partial M_i)
    \ar[r]_*+<.5em>{_{\bigoplus_{i \in I} \comp_{(M_i,\partial M_i)}^n}}^{(*)}
    \ar[l]_-{\cong}
    &
    \bigoplus_{i \in I} H^n(M_i,\partial M_i)
    \ar[r]^-{\cong}_-{\ev}
    & \bigoplus_I \R
    \ar[uu]_-{\sumop}
  }
  \]
  This diagram is commutative: The leftmost and middle squares are
  commutative by functoriality of bounded cohomology and naturality of
  the comparison map. The rightmost square is commutative because one
  can construct a relative fundamental cycle of~$(M,\partial M)$ out
  of relative fundamental cycles of the~$(M_i,\partial M_i)$ plus a
  chain on~$N$ (see proof of Proposition~\ref{prop:ubcupperestimate}).

  By the duality principle and the hypothesis that one of
  the~$\sv{M_i,\partial M_i}$ is non-zero, the arrow~$(*)$
  is non-zero; isolating the corresponding index shows that also
  the composition~$\sumop \circ \ev \circ (*)$ is non-zero.

  In the leftmost square, both horizontal arrows and the upper right
  vertical arrow are isomorphisms by bounded acyclicity of all
  boundary components (and the long exact sequence of pairs in bounded
  cohomology).
  In Section~\ref{subsec:proof graphs}, using graphs of
  groups (Theorem~\ref{thm:graphbac} and
  Example~\ref{exa:graphglueingmfd}), we will show that the left vertical
  arrow~$(**)$ induced by the inclusions~$M_i \hookrightarrow M$ is
  surjective.

  Therefore, the leftmost square shows that $(***)$ is surjective.
  Together with the non-triviality of the composition~$\sumop
  \circ \ev \circ (*)$, we thus obtain that the comparison
  map~$H_b^n(M,\partial M) \to H^n(M,\partial M)$ must be non-zero, as
  desired.
\end{proof}

\subsection{Graphs of groups with boundedly acyclic edge groups}
\label{subsec:proof graphs}

We consider the bounded cohomology of finite graphs of groups with
boundedly acyclic edge groups in relation to the bounded cohomology of
the vertex groups.  We adapt the proof in the amenable case by Bucher
et~al.~\cite{BBFIPP} to the boundedly acyclic situation by using
uniformly boundedly acyclic actions instead of amenable actions.

We first fix basic notation.

\begin{defn}[Graph]
  A \emph{graph} is a tuple~$\graf = (V,E, o, t, \overline\args)$, consisting of a
  set~$V$, a set~$E$, a map~$(o,t) \colon E \to V^2$, and a fixed
  point-free involution~$\overline\args \colon E \to E$ with
  \[ o(e) = t (\overline e)
  \]
  for all~$e \in E$. 
  The elements of~$V$ are called \emph{vertices}, the
  elements of~$E$ are called \emph{edges}.
  The set of \emph{geometric edges} is defined by 
  \[ \overline E \coloneqq \bigl\{ (e,\overline e) \bigm| e \in E \bigr\} \,.
  \]
\end{defn}

\begin{defn}[Graph of groups]
  Let $\graf = (V,E, o,t, \overline \args)$ be a finite graph (i.e.,
  $V$ and $E$ are finite). A \emph{graph of groups~$\gog$ over~$\graf$}
  is a $\graf$-shaped diagram in the category of groups and injective
  group homomorphisms, i.e., $\gog$ consists of the following data:
  \begin{itemize}
  \item A map that associates a group~$G_v$ to each~$v \in V$;
  \item A map that associates a group~$G_e$ to each~$e \in E$
    such that $G_e = G_{\overline e}$;
  \item A map that associates to each edge~$e \in E$ an
    injective group homomorphism~$h_e \colon G_e \to G_{t(e)}$.
  \end{itemize}
\end{defn}

If $G$ is the fundamental group of a graph of groups, then the vertex
and edge groups admit canonical inclusions
into~$G$~\cite[Chapter~5]{serre_trees} and we will identify these
groups with their image inside of~$G$.

We consider finite graphs of groups with boundedly acyclic edge
groups, in analogy with the amenable case~\cite[Theorem~1.1]{BBFIPP};
more precisely:

\begin{setup}\label{setup:graph}
  \hfil
  \begin{itemize}
  \item Let $n \geq 1$;
  \item Let $\graf = (V,E,o,t,\overline{\args})$ be a finite graph;
  \item Let $\gog$ be a graph of groups over~$\graf$;
  \item Let $G$ be the fundamental group of~$\gog$;
    for~$v\in V$, we denote the corresponding inclusion
    by~$i_v \colon G_v \hookrightarrow G$;
  \item The edge groups~$(G_e)_{e \in E}$ are \emph{uniformly boundedly acyclic
    of order~$n$ in~$G$}, i.e., the set
    \[ \biggl\{ \bigcap_{i=1}^{n} g_i G_{e_i} g_i^{-1} 
    \biggm| g_1,\dots, g_n \in G,\ e_1,\dots,e_n \in E
    \biggr\}
    \]
    of subgroups of~$G$ is uniformly boundedly acyclic.
  \end{itemize}
\end{setup}

\begin{thm}\label{thm:graphbac}
  In the situation of Setup~\ref{setup:graph}, let $n \geq 3$ and
  $k \in \{3,\dots, n\}$. Then the map
  \[ \bigoplus_{v \in V} H_b^k(i_v) \colon
  H_b^k(G) \to \bigoplus_{v \in V} H_b^k(G_v)
  \]
  induced by the inclusions is surjective.
\end{thm}

For the proof, we describe the bounded cohomology of~$G$ and the
vertex groups via suitable uniformly boundedly acyclic actions. 
In the situation of Setup~\ref{setup:graph}, we
consider the set
\[ S \coloneqq (G \times V) \sqcup \coprod_{e \in \overline E} G/G_e
\]
with the $G$-action
\begin{itemize}
\item given on~$G \times V$ by left translation on the first factor;
\item given on each~$G/G_e$ by left translation of cosets.
\end{itemize}
In particular, for~$k \in \{0,\dots,n-1\}$, the diagonal action of~$G$
on~$S^{k+1}$ is uniformly boundedly acyclic, since we assumed uniform
bounded acyclicity of order~$n$ for the collection of edge groups.  By
Remark~\ref{rem:alt:computes:bc}, the bounded cohomology of~$G$ is
canonically isomorphic to the cohomology of the
complex~$\Balt(S^{*+1},\R)^G$ in degrees~$\leq n$.

Similarly, for each vertex~$v \in V$, we consider
the $G_v$-set
\[ S_v \coloneqq G_v \sqcup \coprod_{e \in E, t(e) = v} G_v / G_e \,.
\]
with the left translation action.  In the situation of
Setup~\ref{setup:graph}, the diagonal action on~$S_v^{k+1}$ is
uniformly boundedly acyclic for~$k \in \{0,\dots,n-1\}$ and thus the
bounded cohomology of~$G_v$ is canonically isomorphic to the
cohomology of the complex~$\Balt(S_v^{*+1},\R)^{G_v}$ in degrees~$\leq
n$ (Remark~\ref{rem:alt:computes:bc}).

For each vertex~$v \in V$, there is a canonical inclusion~$\varphi_v
\colon S_v \to S$; on the first summand, this is given
by~$\varphi_v(g) \coloneqq (g,v)$ for all~$g \in G_v$, on the other summands,
we use the canonical maps induced by the canonical inclusions~$i_v
\colon G_v \hookrightarrow G$. By construction, $\varphi_v$ is
$G_v$-equivariant with respect to the inclusion~$i_v$.

With this preparation, we give the proof of Theorem~\ref{thm:graphbac}: 

\begin{proof}[Proof of Theorem~\ref{thm:graphbac}]
  We write
  \[ \varphi^* \coloneqq \bigoplus_{v \in V} \varphi_v^*
  \colon \Balt(S^{*+1},\R)^G
  \to
  \bigoplus_{v \in V}\Balt(S_v^{*+1},\R)^{G_v}
  \]
  for the combination of the~$\varphi_v^*$. 
  Bucher et al.~\cite[Theorem~4.1]{BBFIPP} provide a construction of a cochain
  map~$\psi^* \colon \bigoplus_{v \in V} \Balt(S_v^{*+1},\R)^{G_v}
  \to \Balt(S^{*+1},\R)^G$
  in degrees~$\geq 2$ that is right-inverse to~$\varphi^*$.
  Then also~$H^k(\varphi^*)$ has a right inverse in degrees~$\geq 2$. 

  Let $k \in \{3,\dots, n\}$.  
  By Lemma~\ref{lemma:prel:induced:at:level:of:group:is:the:same}
  (applied to~$i_v$ and $\varphi_v$ for each vertex~$v \in V$)
  and using that~$V$ is finite,
  we obtain the following
  commutative diagram:
  \[\text{%
  \makebox[0pt]{%
  \xymatrix{%
    H^k\bigl( \Balt(S^{*+1},\R)^G\bigr)
    \ar[rr]
    \ar[d]_{H^k(\varphi^*)}
    &&
    H_b^k(G;\R)
    \ar[d]_{\bigoplus_{v \in V}{H_b^k(i_v)}}
    \\ 
    H^k\bigl(\bigoplus_{v \in V} \Balt(S_v^{*+1},\R)^{G_v}\bigr)
    \ar@{<->}[r]^-{\cong}
    &
    \bigoplus_{v \in V} H^k\bigl(\Balt(S_v^{*+1},\R)^{G_v}\bigr)
    \ar[r]
    &
    \bigoplus_{v \in V} H_b^k(G_v;\R) \,.
    }}}
  \]
  Here the horizontal maps are the canonical maps.
  As $k \in \{3,\dots,n\}$, these horizontal maps
  are isomorphism and the left vertical arrow
  admits a right inverse (given by~$\psi^*$).
  Therefore, also the right vertical arrow has a right inverse.
  In particular, the right vertical arrow~$\bigoplus_{v \in V} H_b^k(i_v)$
  is surjective.
\end{proof}

In particular, Theorem~\ref{thm:graphbac} applies to the
glueing situation of Theorem~\ref{thm:bacglue}:

\begin{ex}\label{exa:graphglueingmfd}
  In the situation of Theorem~\ref{thm:graphbac}, the fundamental
  group~$\pi_1(M)$ is isomorphic to the fundamental group of a finite
  graph of groups that satisfies the conditions of
  Setup~\ref{setup:graph}. More specifically, the vertex groups
  are isomorphic to the fundamental groups of the~$(M_i)_{i\in I}$ and
  the edge groups are isomorphic to the fundamental groups of
  the boundary components in~$\calN$ along which we glue.
\end{ex}

This concludes the proof of Theorem~\ref{thm:bacglue}.

\appendix
\section{The uniform boundary condition}\label{appx:ubc}

We recall the uniform boundary condition and its basic properties
and consequences in the context of bounded cohomology.
 Moreover, we introduce the uniform uniform boundary condition
and use it to compute the bounded cohomology of bounded products.

\subsection{Normed chain complexes}

We begin with basic terminology for normed chain complexes. A
\emph{normed chain complex} over a normed ring~$R$ is a chain complex
in the category of normed $R$-modules with bounded linear maps;
i.e., the boundary operators in normed chain complexes are degree-wise
bounded linear operators. A \emph{Banach chain complex} is a normed
chain complex over~$\R$ consisting of Banach spaces. Similarly, one has
\emph{normed} [\emph{resp.~Banach}] \emph{cochain complexes}.

Let $C_*$ be a normed chain complex over a normed ring~$R$
with boundary operator~$\partial_*\colon C_*\to C_{*-1}$.
\begin{itemize}
\item We write~$B(C_*,R)$ for the normed cochain complex whose cochain
  modules are the bounded duals of the chain modules of~$C_*$ and
  whose coboundary operators are the duals of the boundary
  operators~$\partial_*$.
\item We write~$\overline C_*$ for the degree-wise completion of~$C_*$
  with the degree-wise continuous extension~$\overline\partial_*$ of
  the boundary operator~$\partial_*$.
\end{itemize}
Over~$\R$, both $B(C_*,\R)$ and $\overline C_*$ are Banach
[co]chain complexes.

If $C_*$ is a normed chain complex, then we obtain an induced
seminorm on~$H_*(C_*)$ via
\[ \| \alpha \|
\coloneqq \inf \bigl\{ |c|
\bigm| \text{$c \in C_*$ is a cycle representing~$\alpha$}
\bigr\}
\]
for all~$\alpha \in H_*(C_*)$. Similarly, this also works for normed
cochain complexes.

\subsection{The uniform boundary condition}\label{subsec:UBC}

The uniform boundary condition asks for uniform control on fillings of
null-homologous [co]cycles in normed [co]chain complexes. In some
contexts, similar properties are encoded in the language of
isoperimetric inequalities. In the following, we will stick to the
terminology of Matsumoto and Morita~\cite{matsumotomorita}.

\begin{defn}[Uniform boundary condition]
  Let $C_*$ be a normed chain complex over a normed ring and let $k
  \in \N$. Then $C_*$ satisfies the \emph{uniform boundary condition
    in degree~$k$} if there exists a constant~$K \in \R_{>0}$ with
  \[ \fa{b \in \im \partial_{k+1} \subset C_k} \exi{c \in C_{k+1}}
     \partial_{k+1} (c) = b \qand |c| \leq K \cdot |b| \,.
  \]
  We abbreviate the uniform boundary condition in degree~$k$ by~$\UBC_k$.
\end{defn}

\begin{rem}\label{rem:ubcco}
  Through re-indexing, we can translate the notion of uniform boundary
  condition also to normed cochain complexes. In this case, we use the
  notation~$\UBC^k$ for the uniform boundary condition in degree~$k$.
  Moreover, all of the results below apply both to chain and cochain
  complexes (with appropriate re-indexing).
\end{rem}

We recall basic inheritance properties of~$\UBC$:

\begin{prop}[Homotopy inheritance of~$\UBC$]\label{prop:ubchinherit}
  Let $k \in \N$ and let $C_*$, $D_*$ be normed chain complexes
  over a normed ring~$R$ that are chain homotopic in the category
  of normed $R$-chain complexes. Then $C_*$ satisfies~$\UBC_k$ if and only
  if~$D_*$ satisfies~$\UBC_k$.
  
  More precisely, let $f_* \colon C_* \to D_*$ and $g_* \colon D_* \to C_*$
  be chain maps that are bounded in each degree and let $h_* \colon D_* \to D_{*+1}$
  be a corresponding degree-wise bounded chain homotopy between~$f_* \circ g_*$
  and~$\id_{D_*}$. If $K \in \R_{>0}$ is a $\UBC_k$-constant for~$C_*$, then
  \[ \|f_{k+1}\| \cdot \|g_k\| \cdot K + \|h_k\|
  \]
    is a $\UBC_k$-constant for~$D_*$.
\end{prop}
\begin{proof}
  Let $x \in \im \partial_{k+1}^D$. Then $g_k(x) \in \im \partial^C_{k+1}$.
  As $K$ is a $\UBC_k$-constant for~$C_*$, there exists~$\widetilde y \in C_{k+1}$
  with
  $\partial^C_{k+1}(\widetilde y) = g_k(x)$
  and $|\widetilde y| \leq K \cdot \bigl|g_k(x)\bigr|$.
  Then
  \[ y \coloneqq f_{k+1} (\widetilde y) - h_k(x)
  \]
  satisfies~$\partial^D_{k+1} (y) = x$ and
  $|y| \leq \bigl(\|f_{k+1}\| \cdot \|g_k\| \cdot K + \|h_k\|\bigr) \cdot |x|
  $, 
  as desired.
\end{proof}

\begin{prop}[Dense subcomplexes and~$\UBC$]\label{prop:ubcdense}
  Let $R$ be a normed ring, let~$D_*$ be a normed chain complex over~$R$,
  and let $C_* \subset D_*$ be a dense subcomplex.
  Let $k \in \N$. 
  Then the following are equivalent:
  \begin{enumerate}[label=\enum]
  \item\label{i:C}
    $C_*$ satisfies~$\UBC_k$;
  \item\label{i:Dbar}
    $\overline D_*$ satisfies~$\UBC_k$ and $\ker \partial_{k+1}^C$
    is dense in~$\ker \partial_{k+1}^{\overline D}$;
  \item\label{i:D}
    $D_*$ satisfies~$\UBC_k$ and $\ker \partial_{k+1}^C$
    is dense in~$\ker \partial_{k+1}^D$.
  \end{enumerate}
\end{prop}
\begin{proof}
  Because $C_*$ is dense in~$D_*$, the completion~$\overline D_*$
  of~$D_*$ is also a completion of~$C_*$. 
  The argument by Matsumoto and Morita~\cite[Theorem~2.8]{matsumotomorita}
  applies to all normed chain complexes and their completions; this shows
  the equivalence of \ref{i:C} and~\ref{i:Dbar}.

  We may apply this to~$D_*$ and its completion~$\overline D_*$, since
  with $\ker \partial_{k+1}^C$ also $\ker \partial_{k+1}^D$ is dense
  in~$\ker \partial_{k+1}^{\overline D}$.
  Hence~\ref{i:Dbar} implies that $D_*$ satisfies~$\UBC_k$; 
  whence~\ref{i:Dbar}
  implies~\ref{i:D}, because $\ker \partial_{k+1}^C$ being dense
  in~$\ker \partial_{k+1}^{\overline D}$ also implies density of~$\ker
  \partial_{k+1}^C$ in~$\ker \partial_{k+1}^D$.

  Also, the argument by Matsumoto and Morita shows that \ref{i:D}
  implies~\ref{i:C}, as this implication does not rely on completeness
  of the ambient complex.
\end{proof}

The following characterisations of~$\UBC$ apply to normed [resp.~Banach]
chain complexes over~$\IR$.

\begin{thm}[{\cite[Theorem~2.8]{matsumotomorita}}]\label{thm:ubccomplete}
  Let $C_*$ be a normed chain complex over~$\R$ and let $k
  \in \N$. Then the following are equivalent:
  \begin{enumerate}[label=\enum]
  \item $C_*$ satisfies~$\UBC_k$;
  \item $\overline C_*$ satisfies~$\UBC_k$ and $\ker \partial_{k+1}$
    is dense in~$\ker \overline\partial_{k+1}$;
  \item The comparison map~$H^{k+1}(B(C_*,\R)) \to H^{k+1}(\Hom_\R(C_*,\R))$
    is injective.
  \end{enumerate}
\end{thm}
\begin{proof}
  The equivalence of the first two items is contained in
  Proposition~\ref{prop:ubcdense}.  The argument by Matsumoto and
  Morita~\cite[Theorem~2.8]{matsumotomorita} for the remaining
  implications applies to all normed chain complexes over~$\R$.
\end{proof}

\begin{thm}[{\cite[Theorem~2.3]{matsumotomorita}}]\label{thm:ubcbanach}
  Let $C_*$ be a Banach chain complex over~$\R$ and let $k \in \N$. Then the
  following are equivalent:
  \begin{enumerate}[label=\enum]
  \item $C_*$ satisfies~$\UBC_k$;
  \item $\im \partial_{k+1}$ is closed in~$C_k$;
  \item $H_k(C_*)$ is a Banach space with respect to the induced seminorm;
  \item $H^{k+1}(B(C_*,\R))$ is a Banach space with respect to the
    induced seminorm.
  \end{enumerate}
\end{thm}
\begin{proof}
  The argument by Matsumoto and Morita~\cite[Theorem~2.3]{matsumotomorita}
  applies to all Banach chain complexes.
\end{proof}

\begin{ex}\label{exa:ubcgroup}
  Let $X$ be a space or a group with~$H_b^*(X;\R) \cong H_b^*(1;\R)$.
  In particular, this bounded cohomology is Banach in all degrees.
  Then the cochain complex version of Theorem~\ref{thm:ubcbanach}
  (Remark~\ref{rem:ubcco}) shows that $C^*_b(X;\R)$ satisfies~$\UBC^k$
  for all~$k \in \N$.
  Moreover, Theorem~\ref{thm:ubccomplete} shows that $C_*(X;\R)$
  satisfies~$\UBC_k$ for all~$k \in \N$.

  This applies to all path-connected spaces
  with amenable fundamental group, to all amenable groups, 
  and to the known boundedly acyclic
  groups~\cite{matsumotomorita,loeh_bcd,fffclmm, fffclmm2, monodnariman, monodthompson}.  In particular,
  there exist finitely presented non-amenable groups~$G$ such that
  $C^*_b(G;\R)$ satisfies~$\UBC^k$ for all~$k \in
  \N$~\cite[Corollary~5.2]{fffclmm}\cite{monodthompson}.
\end{ex}

\begin{rem}
  For the free group~$F_2$ of rank~$2$ it is well known
  that $H^2_b(F_2;\R)$ and $H^3_b(F_2;\R)$ are infinite-dimensional.
  But it is unknown whether the higher bounded cohomology of~$F_2$
  is trivial or not. We outline an approach through the uniform boundary
  condition: 
  Let $k \in \N_{\geq 4}$. 
  Then the following are equivalent:
  \begin{enumerate}[label=\enum]
  \item $H^k_b(F_2;\R) \cong 0$;
  \item $C_*(F_2;\R)$ satisfies~$\UBC_{k-1}$;
  \item $C_*(F_2;\Q)$ satisfies~$\UBC_{k-1}$;
  \end{enumerate}
  Indeed, the first two items are equivalent by
  Theorem~\ref{thm:ubccomplete} and the fact that $H^k(F_2;\R) \cong
  0$.  The equivalence of the last two items follows from
  Proposition~\ref{prop:ubcdense}, the fact that $C_*(F_2;\Q)
  \hookrightarrow C_*(F_2;\R)$ induces an isometric embedding on the
  level of homology~\cite[Proposition~1.7]{loeh_phd}, and a small computation
  using the universal coefficient theorem.
  
  Moreover, the last condition can be reformulated as follows:
  \[ \exi{K \in \R_{>0}}
     \fa{c \in C_k(F_2;\Z)}
     \exi{c' \in C_k(F_2;\Q)}
     \partial_k (c') = \partial_k (c)
     \qand
     |c'|_1 \leq K \cdot |\partial_k (c)|_1 \,.
  \]
  In principle, this allows for experimental testing
  whether~$H_b^k(F_2;\R)$ is trivial or not~\cite{simonlang_bsc}. Of
  course, the main challenge is to efficiently generate large amounts
  of ``interesting'' chains in~$C_k(F_2;\Z)$.
\end{rem}

The uniform boundary condition is useful, e.g., in glueing estimates
for simplicial volume (Section~\ref{subsec:ubcglue}), in vanishing
results for bounded cohomology of certain
groups~\cite{matsumotomorita,loeh_bcd,fffclmm, fffclmm2}, and in vanishing
results for $\ell^1$-homology~\cite{matsumotomorita,fauserloeh_varubc}. 

\subsection{The uniform uniform boundary condition}

We introduce a uniform version of the uniform boundary condition for collections of
normed cochain complexes.

\begin{defn}[Uniform uniform boundary condition]\label{defn:UUBC}
  Let $R$ be a normed ring and let $k \in \N$. A
  collection~$(C_i^*)_{i \in I}$ of normed cochain complexes over~$R$ satisfies
  the \emph{uniform uniform boundary condition in degree~$k$}
  ($\UUBC^k$ for short) if there exists~$K \in \R_{>0}$ that is a
  $\UBC^k$-constant for all~$C_i^*$ with~$i \in I$.
\end{defn}

\begin{ex}[{\cite[Example~4.11]{fffclmm2}}]\label{ex:amenable UUBC}
  We consider a collection~$(H_i)_{i\in I}$ of amenable groups.
  Then $(C_b^*(H_i;\R))_{i \in I}$ satisfies~$\UUBC^k$
  for all~$k \in \N$.
  Indeed, if $H$ is amenable, then $1$ is a $\UBC^k$-constant
  for~$C_b^*(H;\R)$: There exists a contracting cochain
  homotopy~$s$ for~$C_b^*(H;\R)$
  with~$\|s\|\leq1$~\cite[Theorem~3.6]{Frigerio};
  thus, for every cocycle~$b$, the cochain~$c \coloneqq s(b)$
  satisfies
  \[ b = \delta \circ s(b) + s \circ \delta(b)
  = \delta(c)
  \qand
  |c|_\infty \leq \|s\| \cdot |b|_\infty \leq |b|_\infty \,.
  \]
\end{ex}

\begin{ex}\label{exa:uubcgroup}
  Every finite collection of cochain complexes whose members all
  satisfy~$\UBC^k$ satisfies~$\UUBC^k$.

  In particular, from Example~\ref{exa:ubcgroup}, we obtain: If
  $(H_i)_{i \in I}$ is a \emph{finite} collection of boundedly
  acyclic groups, then $(C_b^*(H_i;\R))_{i \in I}$
  satisfies~$\UUBC^k$ for all~$k \in \N$.
\end{ex}

It is unknown whether all collections of boundedly acyclic groups
satisfy~$\UUBC$ in all degrees (as the open mapping theorem used in
the proof of Theorem~\ref{thm:ubcbanach} does not give a priori
estimates on the norms of the partial inverses).  Every collection of
boundedly acyclic groups satisfies $\UUBC$ in
degree~$2$~\cite[Proposition~4.15]{fffclmm2}.

\subsection{Bounded products}\label{subsec:bounded products}

Finite degree-wise products are compatible with taking
cohomology of bounded cochain complexes. For infinite degree-wise
products, in general, one needs to impose boundedness conditions
in a uniform way. To this end, we introduce bounded products
and prove a compatibility statement for cohomology of certain
degree-wise bounded products.

\begin{defn}[Bounded product]
  Let $R$ be a normed ring.
  Let $(V_i)_{i \in I}$ be a collection of normed modules over~$R$.
  The \emph{bounded product} of~$(V_i)_{i \in I}$ is the
  normed $R$-module
  \[ \bprod{i \in I} V_i
  \coloneqq \Bigl\{ x \in \prod_{i \in I} V_i
  \Bigm| \sup_{i \in I} |x_i| < \infty
  \Bigr\} \subset \prod_{i \in I} V_i
  \]
  with respect to the supremum norm~$|\cdot|_\infty$.
\end{defn}

\begin{ex}\label{exa:bprodcoprod}
  Let $(S_i)_{i \in I}$ be a collection of sets and let $R$ be a
  normed ring. Then the canonical inclusions~$(S_j \hookrightarrow
  \coprod_{i \in I} S_i)_{j \in I}$ induce a natural isometry
  \[ \linf \Bigl( \coprod_{i \in I} S_i, R\Bigr)
    \to \bprod{i \in I} \linf(S_i, R)
  \]
  of normed~$R$-modules.
\end{ex}

\begin{rem}[Bounded product of normed cochain complexes]\label{rem:bprodch}
  Let $R$ be a normed ring. A collection~$(C_i^*)_{i \in I}$ of
  normed cochain complexes over~$R$ is \emph{uniform}
  if for each~$k \in \N$, the supremum~$\sup_{i \in I} \|\delta_i^k\|$
  is finite. For example, all collections of normed cochain complexes
  built using simplicial coboundary operators are uniform
  (such as bounded cochain complexes of groups or spaces).

  If $(C_i^*)_{i \in I}$ is a uniform collection of normed cochain
  complexes over~$R$, then the degree-wise bounded product~$(\bprod{i
    \in I} C_i^k)_{k \in \N}$ is a normed cochain complex over~$R$
  with respect to the supremum norm and the degree-wise product
  coboundary operator
  \begin{align*}
    \bprod{i \in I} C_i^* & \to
    \bprod{i \in I} C_i^{*+1}
    \\
    (x_i)_{i \in I} & \mapsto
    (\delta_i^*(x_i))_{i \in I} \,.
  \end{align*}
\end{rem}

\begin{thm}[Cohomology of bounded products]\label{thm:bprod}
  Let $k \in \N$.  Let $(C_i^*)_{i \in I}$ be a uniform collection of
  normed cochain complexes over a normed ring~$R$ that
  satisfies~$\UUBC^k$. Then the map
  \[ \Phi \colon H^k\Bigl( \bprod{i \in I} C_i^* \Bigr)
  \to \bprod{i \in I} H^k(C_i^*)
  \]
  induced by the canonical projections 
  is a continuous isomorphism of $R$-modules with continuous
  inverse. Here, we equip~$H^k(C_i^*)$ with the seminorm induced by
  the given norm on~$C_i^*$.
\end{thm}
\begin{proof}
  Clearly, the map~$\Phi$ is well-defined and continuous. 

  We construct an explicit inverse: Let $\varepsilon \in \R_{>0}$.
  Let $(\varphi_i)_{i \in I} \in \bprod{i \in I} H^k(C_i^*)$; for each~$i \in I$,
  there exists a cocycle~$f_i \in C_i^k$ representing~$\varphi_i$ in~$H^k(C_i^*)$
  with
  \[ |f_i| \leq \|\varphi_i\| + \varepsilon \,.
  \]
  Because $(\varphi_i)_{i \in I}$ lies in the bounded product,
  the norms~$(|f_i|)_{i \in I}$ are a bounded collection, and so
  $f \coloneqq (f_i)_{i\in I} \in \bprod{i \in I} C_i^k$; moreover,
  $\delta (f) = (\delta_i (f_i))_{i\in I} = 0$. 
  Therefore, we obtain a cohomology class
  \[ \varphi \coloneqq [f] \in H^k\Bigl( \bprod{i \in I} C_i^*\Bigr) \,.
  \]
  By construction, $\Phi(\varphi) = (\varphi_i)_{i \in I}$. 

  If this construction is independent of the chosen
  collection~$(f_i)_{i \in I}$, then it provides an $R$-linear inverse
  of~$\Phi$; moreover, as we can take~$\varepsilon \to 0$, we also see
  that this inverse is bounded.
  
  Thus, it remains to show that $\varphi$ is independent of the choice
  of the collection~$(f_i)_{i \in I}$. In order to show this, we use the uniform uniform
  boundary condition: Let $(f'_i)_{i \in I} \in \prod_{i \in I} C_i^k$
  be a collection of cocycles with
  \[ [f'_i] = \varphi_i \in H^k(C_i^*)
     \qand |f'_i| \leq \|\varphi_i\| + 1
  \]
  for all~$i \in I$. Let $K$ be a $\UUBC^k$-constant for~$(C_i^*)_{i \in I}$.
  Then, for each~$i \in I$, there is a cochain~$c_i \in C_i^{k-1}$ with
  \[ \delta_i (c_i) = f_i - f'_i
  \qand |c_i| \leq K \cdot |f_i - f'_i|
  \leq K \cdot 2 \cdot \bigl( \|\varphi_i\| +1\bigr) \,.
  \]
  Thus, $c \coloneqq (c_i)_{i \in I}$ lies in~$\bprod{i \in I} C_i^{k-1}$
  and $\delta (c) = f -f'$. 
\end{proof}


\bigskip

\bibliographystyle{amsalphaabbrv}
\bibliography{bib}

\end{document}